\documentclass[11pt]{article}
\usepackage{amssymb}
\usepackage{mathrsfs}
\usepackage{latexsym,amsmath,amssymb,amsfonts,epsfig,graphicx,cite,psfrag}

\usepackage{pict2e,color,colordvi,amscd}
\usepackage{amsthm}
\usepackage{pstricks}

\usepackage[hang]{subfigure}

\newtheorem{theo}{Theorem}[section]

\newtheorem{lem}[theo]{Lemma}
\newtheorem{coro}[theo]{Corollary}

\numberwithin{equation}{section}
\def\qed{\hfill \rule{4pt}{7pt}}

\begin{document}

\newcommand{\mycir}[1]{\ensuremath{#1^{r}}}
\newcommand{\mycircle}[1]{\ensuremath{\big(#1\big)^{r}}}
\newcommand{\myequation}{\ensuremath{8.956^r + 1.036^r = 10.992^r}}

\title{Circumference of 3-connected cubic graphs}

\author{Qinghai Liu\footnote{email: qliu@fzu.edu.cn; partially
    supported by NSFC Projects 11301086 and 11401103, and the Natural Science Foundation of  Fujian  Province(2014J05004).}
\\Center for Discrete Mathematics\\ Fuzhou University\\ Fuzhou, 350002,
China\\ \medskip \\ Xingxing Yu\footnote{email: yu@math.gatech.edu; partially
  supported by NSF grants DMS-1265564 and DMS-1600738, and the Hundred Talents Program of Fujian Province.}\\School of Mathematics\\ Georgia Institute of Technology\\ Atlanta, GA 30332,\\
  \medskip\\
  Zhao Zhang\footnote{email: hxhzz@sina.com; partially
    supported by NSFC 11531011 and Xinjiang Talent Youth Project 2013711011}\\
College of Mathematics Physics and Information Engineering\\
Zhejiang Normal University\\
Jinhua, Zhejiang, 321004, China}

\maketitle

\begin{abstract}
The circumference of a graph is the length of its longest cycles.
 Jackson established a conjecture of Bondy by showing that the circumference of
 a 3-connected cubic graph of order $n$ is  $\Omega(n^{0.694})$.
Bilinski {\it et al.} improved this lower bound to $\Omega(n^{0.753})$ by
studying large Eulerian subgraphs in 3-edge-connected graphs.
In this paper, we further improve this lower bound to $\Omega(n^{0.8})$.
This is done by  considering certain 2-connected cubic graphs, finding cycles through two given edges, and
distinguishing the cases whether or not these edges are adjacent.
\end{abstract}

\newpage

\section{Introduction}

Tait \cite{tait} conjectured in 1880 that every 3-connected cubic planar graph contains a Hamilton cycle.
This conjecture remained open until a counterexample was found in 1946 by Tutte \cite{tutte}.
There has since been extensive research concerning longest cycles in
graphs, see \cite{bilinski} for more references.
We use $|G|$ to denote the {\it order} of a graph $G$, i.e., the number of vertices in $G$; and
refer to the length of a longest cycle in $G$ as the \textit{circumference} of $G$.
We will be concerned with lower bounds on the circumference of 3-connected cubic graphs.

Barnette \cite{barnette} showed that every 3-connected cubic graph of order $n$ has circumference
$\Omega(\log n)$. Bondy and Simonovits \cite{bondy} improved this bound to $\exp(\Omega(\sqrt{\log n}))$
and conjectured that it can be improved further to $\Omega(n^c)$ for some constant $0<c<1$.
This conjecture was confirmed by Jackson \cite{jackson}, with $c=\log_2(1+\sqrt5)-1 \approx0.694$.
Bondy and Simonovits  \cite{bondy} constructed an infinite family of 3-connected cubic graphs with
circumference $\Theta(n^{\log_98})\approx \Theta(n^{0.946})$.

Recently,
Bilinski, Jackson, Ma and Yu \cite{bilinski} showed that
every 3-connected cubic graph of order $n$ has circumference
$\Omega(n^{\alpha})$, where $\alpha\approx0.753$ is the real root of $4^{1/x}-3^{1/x}=2$.
This is proved by reducing the problem to one about Eulerian subgraphs in
3-edge-connected graphs.

In this paper, we further improve this lower bound by considering certain
vertex weighted, 2-connected cubic graphs (multiple edges allowed). Let $G$ be a graph and let $w: V(G)\rightarrow
\mathbb{Z}^+$, where here $\mathbb{Z}^+$ denotes the set of non-negative
integers. For any $H\subseteq G$, we write $w(H):=\sum_{v\in V(H)}w(v)$.

\begin{theo}\label{thm-main}
Let $r=0.8$ and $c=1/(8^r-6^r)\approx 0.922$.
Let $G$ be a 2-connected cubic graph, let $w: V(G)\rightarrow \mathbb{Z}^+$, and let $e,f\in E(G)$. Suppose
every 2-edge cut in $G$ separates $e$ from $f$. Then there is a cycle $C$ in $G$ with $e,f\in E(C)$ such that
\begin{itemize}
 \item[$(a)$] $w(C)\geq w(G)^r$ when $e,f$ are adjacent, and
 \item[$(b)$] $w(C)\geq cw(G)^r$ when $e,f$ are not adjacent.
\end{itemize}
\end{theo}

We remark that in Theorem~\ref{thm-main}, we may further assume that the weight function $w$ satisfies $w(V(\{e,f\}))=0$.
For otherwise, we define a new weight function $w': V(G)\to \mathbb Z^+$ such that $w'(v)=w(v)$ for all $v\notin V(e)\cup V(f)$
and $w'(v)=0$ for all $v\in V(e)\cup V(f)$. Let $w_0=w(V(e)\cup V(f))$. Then $w(G)=w'(G)+w_0$.
If Theorem~\ref{thm-main} holds for $w'$ then there is a cycle $C$ such that $e,f\in E(C)$ and either $w'(C)\geq w'(G)^r$ or $w'(C)\geq cw'(G)^r$. Thus $w(C)=w'(C)+w_0$ and either $w(C)\geq (w(G)-w_0)^r +w_0\geq w(G)^r$ or $w(C)\geq c(w(G)-w_0)^r+w_0\geq cw(G)^r$. Thus, Theorem~\ref{thm-main} also holds for $w$. Hence, in Sections 3 and 4, we may assume $w$ is a weight function such that
$w(V(\{e,f\}))=0$.

In our proof of Theorem~\ref{thm-main} we divide $G$ into a few smaller
parts, find long cycles in some of these parts,
and merge these cycles into the desired cycle in $G$. The length of the cycle
will be guaranteed by various properties of the function $x^r$ (see
Lemma \ref{lem-inequality}). We will need structural information of
graphs obtained from a 3-connected cubic graph after certain operations,
and we will also need cycles through some specified edges and vertices in
such graphs. Those results are presented in  Section 2.  In Section 3, we
prove Theorem~\ref{thm-main}(a) by inductively applying
Theorem~\ref{thm-main} (both (a) and (b)); and
in Section 4, we prove Theorem~\ref{thm-main}(b) by inductively applying
Theorem~\ref{thm-main} (both (a) and (b)). In Section
5, we complete the proof of Theorem~\ref{thm-main} and give some concluding remarks.

We end this section with notation needed for our presentation.
Let $G$ be a 2-connected cubic graph and $e=uv\in E(G)$. Let $e_1,e_2$ be
the edges of $G-e$ incident with $u$,
and let $e_3,e_4$ be the edges of $G-e$ incident with $v$. Suppose $e_1$ and $e_2$ do not form multiple edges,
and $e_3,e_4$ do not form multiple edges.
We use $G\ominus e$ to represent the graph obtained from $G$
by deleting $e$ and then merging $e_1$ with $e_2$ (equivalently, suppressing
the degree 2 vertex $u$) and merging $e_3$ with $e_4$ (equivalently, suppressing
the degree 2 vertex $v$). So $G\ominus e$ is a cubic graph in which we use
$e_1$, or $e_2$, or $e_1=e_2$ (respectively, $e_3$, or $e_4$, or $e_3=e_4$) to
denote the edge resulting from the merging of $e_1$ and $e_2$ (respectively, $e_3$
and $e_4$).
If $H=G\ominus e$, then we also say $G=H\oplus e$, or $G=(H\oplus e,
e_i,e_j)$ for $i\in \{1,2\}$ and $j\in \{3,4\}$,  i.e.,
$G$ is obtained from $H$ by subdividing $e_i$ and $e_j$ and adding the edge $e$ between the new vertices.

Let $G$ be a graph and $A, B\subseteq G$. We use $[A,B]$ to represent the
set of edges with one end in $A$ and the other in $B$, and write $\partial_G(A):=[A,G-V(A)]$.
If $A$ is connected then we use $G/A$ to represent the graph obtained from $G$
by contracting $A$ to a single vertex (multiple edges are preserved but
loops are deleted).
In general, if $A_1, \dots, A_k\subseteq G$ are disjoint and connected then
$G/(A_1,\ldots,A_k)$ represents the graph obtained from $G$ by
contracting each $A_i$ to a single vertex, for $i=1, \ldots, k$. Note that
if $G$ is cubic and $|\partial_G(A_i)|=3$ for $1\le i\le k$
then $G/(A_1,\ldots,A_k)$ is also cubic.  For convenience, we view any
$U\subseteq V(G)$ as a subgraph of $G$ with vertex set $U$ and no edges.
Also, for any $F\subseteq E(G)$, we let $V(F)$ denote the set of vertices of $G$ incident with $F$.

\section{Useful lemmas}

In our proof of Theorem \ref{thm-main}, the following result will be used frequently.

\begin{lem}\label{lem-submod}
Suppose $A$ and $B$ are subgraphs of a 3-connected cubic graph $G$ with $|\partial_G(A)|
=|\partial_G(B)|=3$. If $V(A\cap B)\neq\emptyset$ and $V(A\cup B)\neq V(G)$ then $|\partial_G(A\cup B)|=3$.
\end{lem}
\begin{proof}
Since $G$ is 3-connected, $|\partial_G(A\cap B)|\geq
3$ and $|\partial_G(A\cup B)|\geq3$. Thus, by the following submodular inequality
$$3\leq|\partial_G(A\cup B)|\leq
|\partial_G(A)|+|\partial_G(B)|-|\partial_G(A\cap B)|\leq3.$$
This forces $|\partial_G(A\cup B)|=3$.
\end{proof}

We will need to find cycles that contain
certain given edges. It is well known that any two edges in a 2-connected
graph are contained in a cycle.
The following lemma due to  Lov\'asz \cite{lovasz} deals with three edges
in 3-connected graphs, see  \cite{aldred} for a proof.

\begin{lem}[Lov\'asz]\label{lem-lovasz}
Let $G$ be a 3-connected graph and $e_1,e_2,e_3$ three distinct edges of
$G$. If $\{e_1,e_2,e_3\}$ is not an edge cut of $G$, then there is a cycle
in $G$ containing  $\{e_1,e_2,e_3\}$.
\end{lem}

We now use Lemmas~\ref{lem-submod} and \ref{lem-lovasz} to prove the following result about
cycles in cubic graphs.

\begin{lem}\label{lem-3-edge}
 Let $G$ be a 3-connected cubic graph, $u\in V(G)$, $e\in E(G)$ with
 $u\notin V(e)$, $N(u)=\{v_1,v_2,v_3\}$, and
 $N(v_i)=\{u,z_i,w_i\}$ for $i=1,2$. Suppose $|\partial_G(X)|\geq 4$ for any
 $X\subseteq G$ with $|V(X)|\geq 2$, $V(e)\not\subseteq V(X)$, $u,v_3\notin V(X)$ and
 $|V(X)\cap\{v_1,v_2\}|=1$. Then for some $k\in\{1,2\}$, there is
 a cycle in $G\ominus uv_k$ containing  $\{e,v_{3-k}v_3, z_kw_k\}$.
\end{lem}

\begin{proof}
It should be noted that, for $i\in \{1,2\}$, if $e=v_iw_i$ or $e=v_iz_i$ then $e=z_iw_i$ in $G\ominus uv_i$.
Let $G_i=G\ominus uv_i$ for $i=1,2$. First, suppose neither $G_1$ nor
$G_2$ is 3-connected. Then, since $G$ is 3-connected,
there exists $X_i'\subseteq G_i$ such that $|\partial_{G_i}(X_i')|=2$,
$w_i,z_i\in X_i'$, and
$v_{3-i},v_3\notin X_i'$. Set $X_i=V(X_i')\cup \{v_i\}$. Then $|X_i|\geq 2$
and $|\partial_G(X_i)|=3$. By the assumption of this lemma,
$e\in E(G[X_i])$ for $i=1,2$;  and hence $X_1\cap
X_2\neq\emptyset$.
By Lemma \ref{lem-submod}, $|\partial_G(X_1\cup X_2)|=3$ and, as a consequence,
$|\partial_G(X_1\cup X_2\cup\{u\})|=2$, a contradiction.

Thus, we may assume without loss of generality that $G_1$ is 3-connected. If $\{e,v_2v_3,z_1w_1\}$
is not an edge cut in $G_1$ or $e=z_1w_1$ in $G_1$ then, by Lemma \ref{lem-lovasz}, there is a
cycle in $G_1$ containing  $\{e,v_2v_3,z_1w_1\}$. So we may assume
that in $G_1$, $e\ne z_1w_1$ and $\{e,v_2v_3,z_1w_1\}$ is an edge cut.
Let $A$ be a component of $G_1- \{e,v_2v_3,z_1w_1\}$ such that $v_3\notin V(A)$. Then
$|\partial_G(A)|=3$, $e\notin E(A)$, $u,v_1,v_3\notin V(A)$, and  $v_2\in
V(A)$. So by the assumption of this lemma, $V(A)=\{v_2\}$. Hence $v_1v_2\in
E(G)$, $e\neq v_1v_2$, and $e$
is incident with $v_2$. However, in this case, $G_2$ and $G_1$ are
isomorphic, as both may be obtained from $G$ by contracting the triangle
$uv_1v_2u$. Hence, $G_2$ is also
3-connected; so  the same argument above shows that $e$ is incident with
$v_1$. Therefore, since $e\neq v_1v_2$, $e$ and $v_1v_2$ are multiple edges between $v_1$ and $v_2$.
However $|\partial_G(\{v_1,v_2,u\})|=1$, a contradiction to the assumption that $G$ is 3-connected.
\end{proof}

\medskip

We also need a lemma concerning properties of the function
$x^r$, which will be used to bound the length of a cycle obtained by merging several cycles.
The parameters in this lemma will represent weights of subgraphs in our proof of
Theorem~\ref{thm-main}.


\begin{lem}\label{lem-inequality}
 Let $t,w,x,y,z\in \mathbb{Z}^+$, let $r=0.8$ and $c=1/(8^r-6^r)\approx 0.922$.
Then the following statements hold.
\begin{itemize}
\item[$(i)$] If $x\geq 8.956z$ and $y\geq 1.036z$ then $x^r+y^r\geq (x+y+z)^r$.
\item[$(ii)$] If $x\leq 10.174y$ then $cx^r+y^r\geq(x+y)^r$.
\item[$(iii)$] If $0.5y\leq x\leq 8.884y$ then $x^r+y^r\geq (1+1/10.174)^r(x+y)^r$.
\item[$(iv)$] If $z<1.98(t+w+x+y)$ and
  $0<t\leq2.072\cdot\min\{w/1.036, x,y,z/5.884\}$ then
  $w^r+x^r+y^r+cz^r\geq(t+w+x+y+z)^r$.
\item[$(v)$] If $w\leq \min\{x,y,z\}$ then $cx^r+y^r+z^r\geq c(w+x+y+z)^r$.
\item[$(vi)$] If $x\geq 6z$ and $y\geq z$ then $cx^r+y^r\geq c(x+y+z)^r$.
\end{itemize}
\end{lem}

\begin{proof}
 Clearly, (i) holds when $z=0$. So we may assume that $z> 0$. Then $x>0$ and
 $y>0$. Let $f(x,y,z):=x^r+y^r-(x+y+z)^r$.
Then the partial derivative $$f_x(x,y,z)=\frac{r}{x^{1-r}}-\frac{r}{(x+y+z)^{1-r}}>0.$$
So $f(x,y,z)$ is increasing with respect to $x$. Similarly, we can show that $f(x,y,z)$ is increasing with
respect to $y$. Hence $$f(x,y,z)\geq f(8.956z,1.036z,z)=z^r(8.956^r+1.036^r-10.992^r)\approx2.918\times 10^{-5}z^r> 0, $$
and (i) holds.

To prove (ii), let $f(x,y):=cx^r+y^r-(x+y)^r$. We may assume $x>0$, as (ii) holds when $x=0$.
Note that $$f_x(x,y)=\frac{cr}{x^{1-r}}-\frac{r}{(x+y)^{1-r}}.$$
So $f_x(x,y)\ge 0$ if and only if $\frac{c^{1/(1-r)}}x\ge  \frac1{x+y}$; and hence
$f_x(x,y)\ge 0$ if and only if $x\le  \alpha y$,
where $\alpha:=\frac{c^{1/(1-r)}}{1-c^{1/(1-r)}}\approx1.983$. Thus, if $x\leq \alpha y$
then $f(x,y)$ is non-decreasing with respect to $x$
and  $f(x,y)\geq f(0,y)=0$. If $\alpha y\leq x\leq 10.174y$
then $f(x,y)$ is decreasing with respect to $x$; so
$$f(x,y)\geq f(10.174y,y)=y^r(10.174^rc+1-11.174^r)\geq 0,$$
and (ii) holds.

For (iii), let $\beta:=(1+1/10.174)^r$ and
$f(x,y):=x^r+y^r-\beta(x+y)^r$. If $y=0$ then $x=0$, and (iii) holds. So we
may assume $y>0$. Hence
$$f_y(x,y)=\frac{r}{y^{1-r}}-\frac{\beta r}{(x+y)^{1-r}}.$$
If $f_y(x,y)\leq 0$ then $1/y<\beta^{1/(1-r)}/(x+y)\approx 1.455/(x+y)<1.5/(x+y)$;
so $x<0.5y$. Thus,  if $x\ge 0.5y$ then $f_y(x,y)> 0$ and
$f(x,y)$ is increasing with respect to $y$; hence, $$f(x,y)\geq f(x,x/8.884)
=\left(\frac{x}{8.884}\right)^r(8.884^r+1-9.884^r\beta)\approx0.0018\left(\frac{x}{8.884}\right)^r\geq0.$$
So (iii) holds.

For (iv), fix $t>0$ and let $f(w,x,y,z):=w^r+x^r+y^r+cz^r-(t+w+x+y+z)^r$. As in the argument for (i),
we can easily show that $f(w,x,y,z)$ is increasing with respect to $w,x$ and $y$. For $z$,
$$f_z(w,x,y,z)=\frac{cr}{z^{1-r}}-\frac{r}{(t+w+x+y+z)^{1-r}}.$$
So $f'_z(w,x,y,z)\le 0$ if and only if $c/z^{1-r}\le 1/(t+w+x+y+z)^{1-r}$ if and only if  $d/(1-d)\le z/(t+w+x+y)$, where $d=c^{\frac{1}{1-r}}$.
Note that $d/(1-d)>1.98$; so if $z<1.98(t+w+x+y+z)$ then $f_x(w,x,y,z)>0$ and,
hence,
\begin{align*}
f(w,x,y,z)&\geq f(1.036t/2.072, t/2.072, t/2.072, 5.884t/2.072)\\
&=(t/2.072)^r(1.036^r+2+5.884^rc-10.992^r)\\
&\approx 0.0275(t/2.072)^r\\
&\geq 0.
\end{align*}
Therefore, we have (iv).

To prove (v), consider $f(w,x,y,z):=cx^r+y^r+z^r-c(w+x+y+z)^r$. Clearly (v)
holds if $w=0$. So  assume $w>0$; hence  $x>0, y>0$ and $z>0$. It is easy to check that
 $f$ is increasing with respect to $x, y, z$.
Hence $f(w,x,y,z)\ge f(w,w,w,w)=(c+2-4^rc)w^r\approx 0.128w^r\ge 0$. This implies (v).

For (vi), we may assume $z>0$ as (vi) holds when $z=0$. Then  $x>0$ and
$y>0$. Let $f(x,y,z):=cx^r+y^r-c(x+y+z)^r$ which is increasing with respect to $x, y$. So
$f(x,y,z)\ge f(6z,z,z)=(6^rc+1-8^rc)z^r=0$ (by the definition of $c$), and (vi) holds.
\end{proof}

As mentioned in Section 1, we will merge  small cycles into a large one. For this, we need the following lemma.

\begin{lem}\label{lem-merge}
Let $n\geq 4$ be an integer and $r,G,e,f,w$ be defined as in Theorem \ref{thm-main}. Assume that Theorem \ref{thm-main}
holds for graphs of order less than $n$. Let $X_1,\ldots, X_s\subseteq V(G)-(V(e)\cup V(f))$ be pairwise disjoint non-trivial sets (i.e. $|X_i|\ge 2$)
and $G'=G/(X_1,\ldots, X_s)$ be cubic. For
$i=1,\ldots, s$, let $v_i$ denote the vertex obtained by contracting $X_i$. Assign weight 0 to all $v_i$ and, for all other vertices of $G'$, keep their weights from $G$.
Let $e_1,\ldots, e_t\in E(G')$. If $G'\ominus e_1 \ominus \ldots \ominus e_t$ has a
cycle $C'$ such that $\partial_G(X_i)\cap E(C')\neq\emptyset$ for $i\in I\subseteq \{1,\ldots, k\}$
then $G$ has a cycle $C$ such that $E(C')\subseteq E(C)$ and $w(C)\geq w(C')+\sum_{i\in I} w(X_i)^r$.
\end{lem}

\begin{proof}
Since $G'$ is cubic and $X_1,\ldots, X_s$ are pairwise disjoint, $|\partial_G(X_i)|=3$ for $1\leq i\leq s$. Let $\partial_G(X_i)=\{f_i,g_i,h_i\}$ and $G_i=G/(G-V(X_i))$ for $1\leq i\leq s$.
If $G_i$ is not 3-connected then there is an $A\subseteq G[X_i]$ such that $|\partial_{G_i}(A)|\leq 2$. Note that
$|\partial_{G}(A)|=|\partial_{G_i}(A)|$.  Since $G$ is 2-connected, $\partial_G(A)$ is a 2-edge cut in $G$, which does not separate $e$ for $f$
(because $e,f\notin E(G[X_i])$), a contradiction. Hence $G_i$ is a 3-connected cubic graph.
In $G_i$, assign weight 0 to the vertex resulting from contracting $G-V(X_i)$ and let all other vertices inherit their weights from $G$. Then Theorem \ref{thm-main} holds for $G_i$.

Let $I=\{1\leq i\leq k : \partial_G(X_i)\cap E(C')\neq\emptyset\}$. Without loss of generality, let $f_i,g_i\in E(C')$ for each $i\in I$.
Applying Theorem \ref{thm-main} to $G_i$, there is a
cycle $C_i$ in $G_i$ such that $f_i,g_i\in E(C_i)$ and $w(C_i)\geq w(X_i)^r$. Let $C:=G[E(C')\cup \bigcup_{i\in I}E(C_i)]$. Then $C$ is a cycle in
$G$ such that $E(C')\subseteq E(C)$ and $w(C)=w(C')+\sum_{i\in I}w(C_i)\geq w(C')+\sum_{i\in I} w(X_i)^r$.
\end{proof}

\section{Adjacent edges}

In this section we prove Theorem~\ref{thm-main}(a) for graphs of order $n$
under the assumption that Theorem~\ref{thm-main}
holds for graphs of order less than $n$. Recall from the remark following Theorem~\ref{thm-main} that we may assume $w(V(\{e,f\}))=0$.

\begin{lem}\label{lem-main-a}
Let $r=0.8$ and $n\ge 4$ be an integer, and assume that Theorem~\ref{thm-main} holds for graphs of order less than $n$.
Let $G$ be a 2-connected cubic graph of order $n$, let $e,f\in E(G)$ such that $e$ and $f$ are adjacent, and
every 2-edge cut in $G$ separates $e$ from $f$, and let
$w: V(G)\rightarrow \mathbb{Z}^+$ such that $w(V(\{e,f\}))=0$.
 Then there is a cycle $C$ in $G$ such that $e,f\in E(C)$ and $w(C)\geq w(G)^r$.
\end{lem}

\begin{proof}
Since $e$ and $f$ are adjacent, there is no 2-edge cut in $G$ separating $e$ from $f$. So by the assumption of this lemma,
$G$ is 3-connected.

\medskip

{\it Claim 1}. We may assume that no nontrivial 3-edge cut in $G$ contains $e$ or $f$.

For, let $S$ be a nontrivial 3-edge cut in $G$ such that $e\in S$.
Let $A,B$ be the components of $G-S$; then $|A|\ge 2$ and $|B|\ge 2$ (since
$S$ is nontrivial).
So $G/A$ and $G/B$ are  3-connected cubic graphs of order less than $n$. By
assumption, Theorem~\ref{thm-main} holds for both $G/A$ and $G/B$.
Assign weight 0 to the new vertices resulting from contracting $A$ and $B$,
and let all other vertices of  $G/A$ and $G/B$
inherit their weights from $G$.

Without loss of generality, we may assume that $f\in E(G/A)$. Since Theorem~\ref{thm-main} holds for $G/A$,
there is a cycle $C_A$ in $G/A$ such that $e,f\in E(C_A)$ and $w(C_A)\ge  w(G/A)^r$.
By Lemma~\ref{lem-merge}, there is a cycle $C$ in $G$ such that
$e,f\in E(C)$ and
$$w(C)\ge w(G/A)^r+w(B)^r  \ge \big(w(B)+w(A)\big)^r = w(G)^r,$$
completing the proof of Claim 1.
\medskip

Let $e=u_1u_2$ and $f=u_2u_3$, let $e_1,e_2$ be the edges of $G-e$ incident with $u_1$, let
$e_3,e_4$ be the edges of $G-f$ incident with $u_3$, and let $e_5$ be the  edge of $G-\{e,f\}$ incident with $u_2$.
See Fig. \ref{fig2}.

\medskip

{\it Claim 2}. We may assume that $\{e_1,e_2\}\cap \{e_3,e_4\}=\emptyset$.

For, suppose $\{e_1,e_2\}\cap \{e_3,e_4\}\ne \emptyset$ and, without loss of generality, assume $e_2=e_3$.
Let $G':=G/G[\{u_1,u_2,u_3\}]$, assign weight 0 to the vertex resulting from the
contraction of $G[\{u_1,u_2,u_3\}]$, and let the other vertices of $G'$ inherit their weights from $G$.
Then $w(G')=w(G)$ as $w(\{u_1,u_2,u_3\})=w(V(\{e,f\}))=0$. Since $|G'|<n$, it follows from the assumption of this
lemma that Theorem~\ref{thm-main} holds for $G'$. So $G'$ has a cycle $C'$ such that $e_1,e_4\in E(C')$ and $w(C')\ge
w(G')^r$. Now $C:=G[E(C')\cup \{e,f\}]$ is a cycle
in $G$ such that $e,f\in E(C)$ and  $$w(C)\ge w(G')^r=w(G)^r,$$
completing the proof of Claim 2.

\medskip

For $1\leq i\leq 5$,  let $X_i\subseteq G$ be maximal subject to the
following conditions: $e_i\in\partial_G(X_i)$,
$|\partial_G(X_i)|=3$, and $\{u_1,u_2,u_3\}\cap V(X_i)=\emptyset$.
Note that such $X_i$ exists,
as $G$ is cubic and $V(e_i)-\{u_1,u_2,u_3\}$ satisfies these conditions. Moreover,
$G[X_i]$ is connected  as $G_i$ is 3-connected.

\medskip
{\it Claim 3}. For each $1\le i\le 5$, $X_i$ is uniquely defined, and
$(G/X_i)\ominus e_i$ is 3-connected.

First, let $X_i'\subseteq G$ be maximal such that $e_i\in\partial_G(X_i')$,
$|\partial_G(X_i')|=3$, and $\{u_1,u_2,u_3\}\cap V(X_i')=\emptyset$.
By definition, $X_i\cap X_i'\ne \emptyset$ and $X_i\cup X_i'\ne G$. Hence,
since $G$ is 3-connected, it follows from  Lemma \ref{lem-submod} that
$|\partial_G(X_i\cup X_i')|= 3$. By the maximality of $X_i$ and
$X_i'$, we have $X_i=X_i'$.

Now suppose $(G/X_i)\ominus e_i$ is not 3-connected for some $1\le i\le
5$, and let $F$ be a 2-edge cut in $(G/X_i)\ominus e_i$.
Then $F\cup \{e_i\}$ is a 3-edge cut in $G$ and $X_i$ is properly contained in a component of $G-(F\cup \{e_i\})$.
By Claim 1, this component of $G-(F\cup \{e_i\})$
containing $X_i$ is disjoint from $\{u_1,u_2,u_3\}$, contradicting the maximality of $X_i$.

\medskip
{\it Claim 4}.  $X_1\cap X_2=\emptyset$, $X_3\cap X_4=\emptyset$, and $X_5\cap X_i=\emptyset$ for $1\le i\le 4$.

Suppose $X_1\cap X_2\neq\emptyset$. Then as $G$ is 3-connected and $X_1\cup X_2\ne G$, it follows from Lemma \ref{lem-submod}
that $|\partial_G(X_1\cup X_2)|=3$. Hence $X_1=X_2$ by the maximality of $X_1$ and $X_2$.
Therefore, $e_1,e_2\in \partial_G(X_1)$ and, thus, $|\partial_G(X_1\cup\{u_1\})|=2$, contradicting the fact that
$G$ is 3-connected. So $X_1\cap X_2=\emptyset$. Similarly, we have $X_3\cap X_4=\emptyset$.

Now suppose $X_5\cap X_i\neq\emptyset$ for some $1\le i\le 4$. By symmetry,
we may assume $i=1$. Then by Lemma \ref{lem-submod} and by the maximality of $X_1$ and $X_5$,  we can show
that $X_1=X_5$ and $e_1,e_5\in \partial_G(X_1)$.
By Claim 2, the 3-edge cut
$\partial_G(X_1\cup \{u_1,u_2\})$ is nontrivial; but it contains $f$,
contradicting Claim 1.
This completes the proof of Claim 4.

\medskip

{\it Claim 5.} $[X_1,X_2]=\emptyset$ and $[X_3, X_4]=\emptyset$.

If $[X_1,X_2]\neq\emptyset$ then
$\partial_G(X_1\cup X_2\cup\{u_1\})$ is a nontrivial 3-edge cut containing
$e$, a contradiction to Claim 1. Hence $[X_1,X_2]=\emptyset$. Similarly,
$[X_3,X_4]=\emptyset$, completing the proof of Claim 5.

\medskip

By Claim 4, we have two cases to consider: $(X_1\cup X_2)\cap (X_3\cup
X_4)=\emptyset$ and $(X_1\cup X_2)\cap (X_3\cup X_4)\ne \emptyset$.
 Note that for distinct $X_i$ and $X_j$, $|[X_i,X_j]|\le 1$ as $G$ is 3-connected. However, we will see that  it is possible to have  $|[X_i,X_j]|=1$.

\medskip

{\it Case 1.} $(X_1\cup X_2)\cap (X_3\cup X_4)=\emptyset$.

In this case, we have from Claim 4 that $X_i\cap X_j=\emptyset$ for $1\le i<j\le 5$.
Let $x_i:=w(X_i)$ and $\partial_G(X_i)=\{e_i,e_{i1},e_{i2}\}$ for $1\le i\le 5$ (see Fig. \ref{fig2}), and let
$Z:=G-\bigcup_{i=1}^5X_i$ and $z:=w(Z)$. Then
$$w(G)=x_1+x_2+x_3+x_4+x_5+z.$$
By symmetry,
we may assume that $x_1=\min\{x_1,x_2,x_3,x_4\}$ and  $x_3\leq x_4$.
We will find several cycles in $G$ and show that one of these
is the desired cycle.

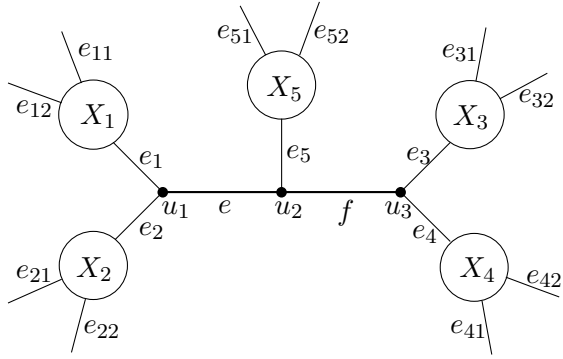
\begin{figure}[ht]\centering
\begin{picture}(208.5,133.5)
\put(148.5,61.5){\circle*{4}}
\put(103.5,61.5){\circle*{4}}
\put(58.5,61.5){\circle*{4}}
\put(115.5,124.5){\makebox(0,0)[tl]{$e_{52}$}}
\put(79,124.5){\makebox(0,0)[tl]{$e_{51}$}}
\put(196.0,31.8){\makebox(0,0)[tl]{$e_{42}$}}
\put(167.3,12.8){\makebox(0,0)[tl]{$e_{41}$}}
\put(193.0,100.3){\makebox(0,0)[tl]{$e_{32}$}}
\put(164.0,117.3){\makebox(0,0)[tl]{$e_{31}$}}
\put(28.5,12.8){\makebox(0,0)[tl]{$e_{22}$}}
\put(3.0,33.8){\makebox(0,0)[tl]{$e_{21}$}}
\put(3.0,96.3){\makebox(0,0)[tl]{$e_{12}$}}
\put(26.3,117.8){\makebox(0,0)[tl]{$e_{11}$}}
\put(97.8,105.3){\makebox(0,0)[tl]{$X_5$}}
\put(170.8,36.8){\makebox(0,0)[tl]{$X_4$}}
\put(168.3,94.5){\makebox(0,0)[tl]{$X_3$}}
\put(25.8,37.0){\makebox(0,0)[tl]{$X_2$}}
\put(27.3,95.0){\makebox(0,0)[tl]{$X_1$}}
\put(105.3,79.3){\makebox(0,0)[tl]{$e_5$}}
\put(152.8,48.8){\makebox(0,0)[tl]{$e_4$}}
\put(150.5,78.5){\makebox(0,0)[tl]{$e_3$}}
\put(49.5,50){\makebox(0,0)[tl]{$e_2$}}
\put(49.5,77.0){\makebox(0,0)[tl]{$e_1$}}
\put(124.5,58.5){\makebox(0,0)[tl]{$f$}}
\put(79.5,58.5){\makebox(0,0)[tl]{$e$}}
\put(142.3,58.0){\makebox(0,0)[tl]{$u_3$}}
\put(100.8,58.0){\makebox(0,0)[tl]{$u_2$}}
\put(58.0,58.0){\makebox(0,0)[tl]{$u_1$}}
\qbezier(188.3,29.3)(198.4,25.5)(208.5,21.8)
\qbezier(179.3,20.3)(181.1,10.1)(183.0,0.0)
\qbezier(186.0,96.8)(195.0,101.6)(204.0,106.5)
\qbezier(177.0,103.5)(178.9,113.6)(180.8,123.8)
\qbezier(110.3,113.3)(114.0,123.4)(117.8,133.5)
\qbezier(97.5,113.3)(92.6,122.6)(87.8,132.0)
\qbezier(29.3,21.0)(26.6,10.9)(24.0,0.8)
\qbezier(20.3,28.5)(10.1,24.0)(0.0,19.5)
\qbezier(27.8,102.8)(24.0,112.5)(20.3,122.3)
\qbezier(148.5,61.5)(157.9,52.1)(167.3,42.8)
\qbezier(148.5,61.5)(157.9,70.9)(167.3,80.3)
\qbezier(103.5,61.5)(103.5,75.4)(103.5,89.3)
\qbezier(58.5,61.5)(49.5,52.5)(40.5,43.5)
\put(176.3,33.0){\circle{25.7}}
\qbezier(20.3,95.3)(10.1,99.0)(0.0,102.8)
\put(174.8,90.8){\circle{25.8}}
\put(103.5,102.0){\circle{25.8}}
\put(32.3,33.8){\circle{25.8}}
\qbezier(58.5,61.5)(49.1,70.9)(39.8,80.3)
{\linethickness{0.8pt}
\qbezier(103.5,61.5)(126.0,61.5)(148.5,61.5)
\qbezier(58.5,61.5)(81.0,61.5)(103.5,61.5)}
\put(32.3,90.8){\circle{26.2}}
\end{picture}
 \caption{\label{fig2} An illustration for Case 1.}
\end{figure}

Let $G_{12}:=G/X_1\ominus e_1/X_2$ (with order of operation from left to right),
let $v_2$ denote the vertex resulting from the contraction of $X_2$, and
let $w(v_2)=0$ and all other vertices of $G_{12}$
inherit their weights from $G$.
Note that
$$w(G_{12})=w(G)-x_1-x_2=x_3+x_4+x_5+z.$$ 
Also note that
by Claim 3, $G/X_1\ominus e_1$ 
is 3-connected; so $G_{12}$ is 3-connected.
Hence by assumption, Theorem~\ref{thm-main} holds for $G_{12}$.

Thus $G_{12}$ has a cycle $C_1$ such that $e=e_2,f\in E(C_1)$ and
$w(C_1)\ge \mycir{w(G_{12})}$. Note that $C_1$ goes through the vertex representing the contraction of $X_2$.
By Lemma \ref{lem-merge}, there is a cycle $C_{12}$ in $G$ such that
$e,f\in E(C_{12})$ and
\begin{equation}\label{eq-aim-x2}
 w(C_{12})\geq w(G_{12})^r+x_2^r=(x_3+x_4+x_5+z)^r+x_2^r.
\end{equation}

Let $G_{52}:=G/X_5\ominus e_5/X_2$, assign wight 0 to the vertex resulting
from the contraction of $X_2$, and let all other vertices of $G_{52}$ inherit their weights from $G$.
Similarly, let $G_{54}:=G/X_5\ominus e_5/X_4$,
assign weight 0 to the vertex resulting from the contraction of $X_4$, and
let all other vertices of $G_{54}$ inherit
their weights from $G$. Note that
$$w(G_{52})=w(G)-x_2-x_5=x_1+x_3+x_4+z$$ and
$$w(G_{54})=w(G)-x_4-x_5=x_1+x_2+x_3+z.$$
By an argument similar to that for $G_{12}$, we see that both $G_{52}$ and $G_{54}$ are 3-connected,
and we can find a cycle in $G_{52}$ through $e=f$ and $e_2$ (and the vertex representing the contraction of $X_2$), and a cycle in $G_{54}$ through $e=f$ and $e_4$
(and the vertex representing the contraction of $X_4$).
By Lemma~\ref{lem-merge}, there exist two cycles $C_{52}$ and $C_{54}$ in $G$
such that $e,f\in E(C_{52})\cap E(C_{54})$ and
\begin{align}
 w(C_{52})\geq w(G_{52})^r+x_2^r=(x_1+x_3+x_4+z)^r+x_2^r, \label{eq-aim-x5-1}\\
 w(C_{54})\geq w(G_{54})^r+x_4^r=(x_1+x_2+x_3+z)^r+x_4^r.\label{eq-aim-x5-2}
\end{align}

Let $G_5:=G/X_5\ominus e_5$. Denote by $e',f'$ the edges obtained from
merging $e$ and $f$ and merging $e_{51}$ and $e_{52}$, respectively.
Then $G_5$ is 3-connected (by Claim 3), and
$$w(G_5)=w(G)-x_5=x_1+x_2+x_3+x_4+z.$$
By assumption of this lemma, Theorem~\ref{thm-main} holds for $G_5$; so
$G_5$ has a cycle $C_1$ such that $e',f'\in E(C_1)$
and $w(C_1)\ge cw(G_5)^r$ (here $e'$ and $f'$ are not adjacent).
By Lemma \ref{lem-merge}, there is a cycle $C_5$ in $G$ such that $e,f\in E(C_5)$ and
\begin{equation}\label{eq-aim-x5-3}
 w(C_5)\geq cw(G_5)^r+x_5^r=c(x_1+x_2+x_3+x_4+z)^r+x_5^r.
\end{equation}

Let $H=G/(X_1,X_2,X_3,X_4,X_5)-\{u_1,u_2,u_3\}$ and $S=\{v_1,v_2,v_3,v_4,v_5\}$, with
each $v_i$ representing the contraction of $X_i$.

We  may assume $V(H)\neq S$. For, if $V(H)=S$ then $H\cong C_5$ because $|[X_i,X_j]|\leq 1$ for all $i\neq j$.
Hence, $H$ has an edge  between $\{v_1,v_2\}$ and $\{v_3,v_4\}$.
Without loss of generality, we may assume that $v_1v_3\in E(H)$. Then $(H-v_1v_3)\cup v_1u_1u_2u_3v_3$ is a hamiltonian cycle
in $G/(X_1,X_2,X_3,X_4,X_5)$. By Lemma \ref{lem-merge},
$G$ has a cycle $C$ such that  $e,f\in E(C)$ and
$$w(C)\geq x_1^r+x_2^r+x_3^r+x_4^r+x_5^r\geq (x_1+x_2+x_3+x_4+x_5)^r=w(G)^r,$$
and the assertion of the lemma holds.

Let $G_z$ be the graph obtained from $H$ by suppressing all vertices of degree 2 (i.e., vertices in $S$).
Since $V(H)\neq S$, $G_z$ is cubic. For any $A\subseteq G_z$, let $A^H$ be the subgraph of $H$ obtained from $A$ by
un-suppressing the vertices in $S$.

We claim that $G_z$ is 2-connected. Otherwise, $G_z$ has disjoint induced subgraphs $Z_1, Z_2$ such that
$V(G_z)=V(Z_1)\cup V(Z_2)$ and $|[Z_1,Z_2]|\le 1$. Without loss of generality we may assume that
$|V(Z_1^H)\cap S|\leq2$. Since $G$ is 3-connected,  $|V(Z_1^H)|\ge 3$, $|V(Z_1^H)\cap S|=2$ and $|[Z_1,Z_2]|= 1$. If
we let $V(Z_1^H)\cap S=\{v_k,v_l\}$ and $[Z_1,Z_2]=\{g\}$ then $\{e_k,e_l,g\}$ is a 3-edge cut in $G$ contradicting
the maximality of $X_k$ and $X_l$.

Let $A_1,\ldots, A_t$ denote the minimal induced subgraphs of $G_z$
such that $|\partial_{G_z}(A_i)|=2$.
We claim that if $t>0$ then $|V(A_i^H)\cap S|\geq2$ for all $i$, and thus $t\leq 2$. For, if
there exists some $A_i\subseteq G_2$ with $V(A_i^H)\cap S=\{v_s\}$ and $\partial_H(A_i^H)=\{g_1,g_2\}$ then
$\partial_G(A_i^H)=\{e_s,g_1,g_2\}$ is a 3-edge cut in $G$,
contradicting the maximality of $X_s$.

Therefore, since $t\le 2$, there exist $k\in \{1,2\}$ and $l\in\{3,4\}$
such that $(G_z\oplus h,e_{k1},e_{l1})$
is 3-connected. So Theorem~\ref{thm-main} holds for $G_z\oplus h$. Let the
vertices of $G_z$ inherit their weights from $G$, and let the new vertices
of $G_z\oplus h$ have weight 0. By Claim 5, $e_{11}\neq e_{21}$ and $e_{31}\neq e_{41}$.
Then for any $i\in\{3-k,7-l,5\}$, there is a cycle $D_i$ in $G_z\oplus h$ such that $h,e_{i1}\in E(D_i)$
and $$w(D_i)\geq cw(G_z\oplus h)^r=cz^r.$$
Note that $D_i$ goes through the vertices representing the contractions of $X_i,X_k,X_l$.
By Lemma~\ref{lem-merge}, choosing $i\in\{3-k,7-l,5\}$ to maximize $x_i$, there is a cycle $C_z$ in $G$ such that
$e,f\in E(C_z)$ and
$$w(C_z) \geq  x_k^r+x_l^r+cz^r+\max\{x_{3-k}^r,x_{7-l}^r,x_5^r\}.$$
Since $x_1\leq x_2$ and $x_3\leq x_4$ and because $k\in \{1,2\}$ and $l\in \{3,4\}$, we
have
\begin{equation}\label{eq-aim-z}
 w(C_z)\geq x_1^r+x_3^r+cz^r+\max\{x_2^r,x_4^r,x_5^r\}.
\end{equation}

We now show that
$$\max\{w(C_{12}),w(C_{52}),w(C_{54}),w(C_5),w(C_z)\}\geq w(G)^r=(x_1+x_2+x_3+x_4+x_5+z)^r.$$
In view of \eqref{eq-aim-x2}--\eqref{eq-aim-z}, let
\begin{align*}
f_1(x_1,x_2,x_3,x_4,x_5,z) &:= (x_3+x_4+x_5+z)^r+x_2^r-(x_1+x_2+x_3+x_4+x_5+z)^r,\\
f_2(x_1,x_2,x_3,x_4,x_5,z) &:= (x_1+x_3+x_4+z)^r+x_2^r-(x_1+x_2+x_3+x_4+x_5+z)^r,\\
f_3(x_1,x_2,x_3,x_4,x_5,z) &:= (x_1+x_2+x_3+z)^r+x_4^r-(x_1+x_2+x_3+x_4+x_5+z)^r,\\
f_4(x_1,x_2,x_3,x_4,x_5,z) &:= c(x_1+x_2+x_3+x_4+z)^r + x_5^r-(x_1+x_2+x_3+x_4+x_5+z)^r,\\
f_5(x_1,x_2,x_3,x_4,x_5,z) &:= x_1^r+x_3^r + x_2^r + cz^r -(x_1+x_2+x_3+x_4+x_5+z)^r,\\
f_6(x_1,x_2,x_3,x_4,x_5,z) &:= x_1^r+x_3^r + x_4^r+ cz^r  -(x_1+x_2+x_3+x_4+x_5+z)^r,\\
f_7(x_1,x_2,x_3,x_4,x_5,z) &:= x_1^r+x_3^r  + x_5^r+ cz^r -(x_1+x_2+x_3+x_4+x_5+z)^r.\\
\end{align*}
By \eqref{eq-aim-x2}--\eqref{eq-aim-z}, it suffices to show that for all $x_1,x_2,x_3,x_4,x_5,z\in \mathbb{Z}^+$,
we have $\max\{f_i: 1\le i\le 7\}\geq 0$.

Suppose this is not true.
Then there exist $x_1,x_2,x_3,x_4,x_5,z\in \mathbb{Z}^+$ such that $f_i(x_1,x_2,x_3,$ $x_4,x_5,z)<0$ for $1\le i\le 7$.
Thus $x_1>0$ as otherwise $f_1\geq 0$. So $x_i\ge x_1>0$ for $i=2,3,4$. Also $x_5>0$ as otherwise $f_3\ge 0$.
Hence we may take partial derivatives of $f_i$ (for $1\le i\le 7$) with respect to any variable, and obtain information about
monotonicity of $f_i$.  Note that each $f_i$ is continuous at $x_j=0$ ($1\le j\le 5$) and $z=0$.

Suppose $z\geq 1.98(x_1+x_2+x_3+x_4+x_5)$.  Since $x_i\geq x_1$ for $i=2,3,4$ and $x_5\geq 0$, we have
$z\geq 7.92x_1$; and since $f_1$ is increasing with respect to each of
$z,x_2,x_3,x_4,x_5$,  we have
$f_1\geq f_1(x_1,x_1,x_1,x_1,0,7.92x_1)=(9.92x_1)^r+x_1^r-(11.92x_1)^r\approx0.00775x_1^r\geq 0$, a contradiction.

Hence, $z<1.98(x_1+x_2+x_3+x_4+x_5)$. Then, since  $f_i<0$ for
$i=5,6,7$, it follows from Lemma~\ref{lem-inequality}(iv)  that, for
each permutation  $jkl$ of $\{2,4,5\}$,
$$x_j+x_k>2.072\cdot \min\{x_l/1.036, x_1, x_3, z/5.884\}.$$
If we choose $j,k,l$ so that $x_j\leq x_k\leq x_l$ then, since
 $2.072x_l/1.036=2x_l\ge x_j+x_k$, $x_j+x_k>2.072\cdot\min\{x_1,x_3,z/5.884\}$. Since $x_1\leq x_3$,
\begin{equation}\label{eq-x2x4x_5}
\min\{x_4+x_5,x_5+x_2,x_2+x_4\}>2.072\cdot\min\{x_1,z/5.884\}
\end{equation}
Moreover, since  $f_4<0$, it follows from
Lemma~\ref{lem-inequality}(ii) that
\begin{equation}\label{eq-x_5}
 x_1+x_2+x_3+x_4+z>10.174x_5.
\end{equation}

Suppose $z/5.884\ge x_1$. Then by \eqref{eq-x2x4x_5}, $\min\{x_4+x_5,x_5+x_2,x_2+x_4\}>2.072x_1$. Hence
$x_3+x_4+x_5+z\geq(1+2.072+5.884)x_1=8.956x_1$. If $x_2\geq 1.036x_1$ then by Lemma \ref{lem-inequality}(i),
we have $f_1\geq 0$, a contradiction. So $x_2<1.036x_1$. Then, since
$x_2+x_5>2.072x_1$, $x_5>1.036x_1$.
By \eqref{eq-x_5}, $x_3+x_4+z>10.174x_5-x_1-x_2>8.5x_1$. Since $f_1$ is increasing
with respect to $x_3+x_4+z$, $x_2$ and $x_5$,
$$f_1\geq (8.5x_1+1.036x_1)^r+x_1^r-(2x_1+8.5x_1+1.036x_1)^r=(9.536^r+1-11.536^r)x_1^r\geq0,$$ a contradiction.

Hence, $z/5.884<x_1.$  By  \eqref{eq-x_5}, we have
\begin{equation*}
 f_2>g_2:=(x_1+x_3+x_4+z)^r+x_2^r-(1+1/10.174)^r(x_1+x_2+x_3+x_4+z)^r.
\end{equation*}

Suppose $x_4<x_2$. Then, since $z/5.884<x_1$ and $x_1\le x_3\le x_4$,
we have $x_1+x_3+x_4+z< 8.884 x_2$.  Since $g_2<f_2<0$,
it follows from  Lemma~\ref{lem-inequality}(iii)  that $0.5x_2>x_1+x_3+x_4+z\geq 3x_1$.
Hence $x_2\ge 6x_1$, and so
$f_1\geq f_1(x_1,6x_1,x_1,x_1,0,0)\geq (2x_1)^r+(6x_1)^r-(9x_1)^r\approx0.134x_1^r\geq0,$ a contradiction.

Hence, $x_2\leq x_4$. By \eqref{eq-x_5} we have
\begin{equation*}
 f_3\geq g_3:= (x_1+x_2+x_3+z)^r+x_4^r-(1+1/10.174)^r(x_1+x_2+x_3+x_4+z)^r.
\end{equation*}
Since $z/5.884<x_1$, $x_1\le x_2\le x_4$ and $x_3\le x_4$, we have
$x_1+x_2+x_3+z< 8.884 x_4$. Since $g_3\le f_3<0$, it follows from Lemma~\ref{lem-inequality}(iii)
that  $0.5x_4>x_1+x_2+x_3+z\geq 3x_1$. Hence $x_4\ge 6x_1$.
If $x_2\geq 2x_1$ then $f_1\geq f_1(x_1,2x_1,x_1,6x_1,0,0) =(7x_1)^r+(2x_1)^r-(10x_1)^r\approx 0.175x_1^r\geq0$,
a contradiction. So $x_2<2x_1$. If $x_5\leq 1.5x_1$ then $f_3\geq f_3(x_1,x_1,x_1,6x_1,1.5x_1,0)
=(3x_1)^r+(6x_1)^r-(10.5x_1)^r\approx 0.04x_1^r\geq 0$, a contradiction. So $x_5>1.5x_1$. Then from \eqref{eq-x_5} we deduce
$x_3+x_4+x_5+z\geq 11.174x_5-x_1-x_2>13.76x_1$. So $f_1\geq
(13.76x_1)^r+x_1^r-(15.76x_1)^r\approx0.066x_1^r\geq0$, a contradiction.

\medskip
{\it Case 2.} $(X_1\cup X_2)\cap (X_3\cup X_4)\neq\emptyset$.

Without loss of generality, we may assume that $X_2\cap X_4\neq\emptyset$.
Then by Lemma \ref{lem-submod} and by the maximality of $X_2$ and $X_4$, we have $X_2=X_4$.

We may assume that $X_1\cap X_3=\emptyset$. For, suppose
$X_1\cap X_3\neq\emptyset$. Then $X_1=X_3$ by Lemma \ref{lem-submod} and by the maximality of $X_1$ and $X_3$.
Let $\partial_G(X_1)=\{e_1,e_3,f_1\}$ and $\partial_G(X_2)=\{e_2,e_4,f_2\}$, and let $U:=G[X_1\cup X_2\cup \{u_1,u_2,u_3\}]$.
Then $\partial_G(U)=\{e_5,f_1,f_2\}$. So $G/U$ is 3-connected and
Theorem~\ref{thm-main} holds for $G/U$. Assign weight 0 to the vertex
of $G/U$ resulting from the contraction of $U$, and let all other vertices of $G/U$ inherit their weights from $G$.
Then $G/U$ has a cycle $C_1$ such that $f_1,f_2\in E(C_1)$ and $w(C_1)\ge \mycircle{w(G)-w(U)}$.
Note that $C_1$ goes through the vertices representing the
contractions of $X_1,X_2$. Hence, by Lemma~\ref{lem-merge},
there is a cycle $C$ in $G$ such that $e,f\in E(C)$ and
$w(C)\ge \mycir{w(X_1)}+\mycir{w(X_2)}+\mycircle{w(G)-w(U)}\ge \mycir{w(G)}.$

Relabel $X_2=X_4$ as $Y_1$ and let $\partial_G(Y_1)=\{e_2,e_4,e_6\}$.
Let $Y_2\subseteq G-Y_1-\{u_1,u_2,u_3\}$ be maximal such that $e_6\in \partial_G(Y_2)$ and $|\partial_G(Y_2)|=3$.
Note that $Y_2$ exists (as $G$ is cubic) and $G[Y_2]$ is connected as $G$ is 3-connected. Moreover,
$Y_2$ is uniquely defined by the same argument for Claim 3 (that $X_i$
is uniquely defined).

We claim that $Y_2\cap (X_1\cup X_3)=\emptyset$.
For, suppose $Y_2\cap X_i\ne \emptyset$ for some $i\in \{1,3\}$. Then
by Lemma~\ref{lem-submod} and the maximality of $Y_2$ and $X_i$,
we have $Y_2=X_i$. Thus $[X_i, Y_1]\neq\emptyset$, contradicting Claim 5.

We may assume that $Y_2\cap X_5=\emptyset$.
For, if $Y_2\cap X_5\neq\emptyset$ then $Y_2=X_5$ by Lemma \ref{lem-submod} and the maximality of $X_5$ and $Y_2$.
Let $\partial_G(Y_2)=\{e_5,e_6,f_1\}$, and $U=G[Y_1\cup Y_2\cup\{u_1,u_2,u_3\}]$. Then $\partial_G(U)=\{e_1,e_3,f_1\}$.
So $G/U$ is 3-connected, and Theorem~\ref{thm-main} holds for $G/U$. Assign weight 0 to the vertex of $G/U$ resulting from the
contraction of $U$, and let all other vertices of $G/U$ inherit their weights from $G$.
Now $G/U$ has a cycle $C_1$ such that $e_1,f_1\in E(C_1)$
and $w(C_1)\ge \mycircle{w(G)-w(U)}$.
Note that $C_1$ goes through the vertices representing the
contractions of $Y_1,Y_2$. By Lemma~\ref{lem-merge},
there is a cycle $C$ in $G$ such that $e,f\in E(C)$ and $w(C)\ge
\mycircle{w(G)-w(U)}+\mycir{w(Y_1)}+\mycir{w(Y_2)}\ge
\mycir{w(G)}.$

Thus, $G$ has the structure described in  Fig. \ref{fig3}.
Let $x_i:=w(X_i)$ for $i=1,3,5$, $y_i:=w(Y_i)$ for $i=1,2$,  $Z:=G-(X_1\cup X_3\cup X_5\cup Y_1\cup Y_2)$, and
$z:=w(Z)$. Then $$w(G)=x_1+x_3+x_5+y_1+y_2+z.$$
Let $\partial_G(X_i)=\{e_i,e_{i1},e_{i2}\}$ for $i=1,3,5$,
and let $\partial_G(Y_2)=\{e_6,e_{61},e_{62}\}$. By symmetry, we may assume that $x_1\leq x_3$.
We will produce several cycles in $G$ and show that one of these is the desired cycle.

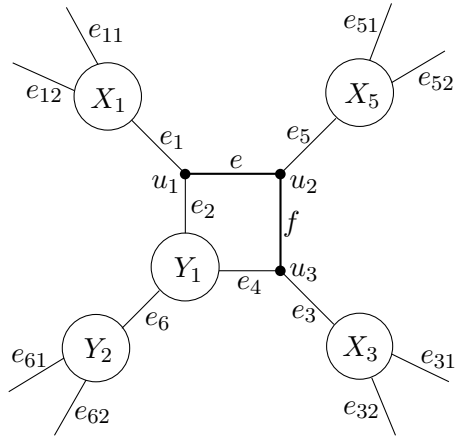
\begin{figure}[ht]\centering
\begin{picture}(166.5,163.5)
\put(102.8,62.3){\circle*{4}}
\put(102.8,99.0){\circle*{4}}
\put(66.8,99.0){\circle*{4}}
\put(24.8,11.0){\makebox(0,0)[tl]{$e_{62}$}}
\put(1.0,32.5){\makebox(0,0)[tl]{$e_{61}$}}
\put(154.8,137.3){\makebox(0,0)[tl]{$e_{52}$}}
\put(126.5,159.0){\makebox(0,0)[tl]{$e_{51}$}}
\put(126.5,12.8){\makebox(0,0)[tl]{$e_{32}$}}
\put(155.8,32.0){\makebox(0,0)[tl]{$e_{31}$}}
\put(6.5,132.5){\makebox(0,0)[tl]{$e_{12}$}}
\put(30.0,155.5){\makebox(0,0)[tl]{$e_{11}$}}
\put(28.5,36.8){\makebox(0,0)[tl]{$Y_2$}}
\put(62.5,67.3){\makebox(0,0)[tl]{$Y_1$}}
\put(127.0,133.3){\makebox(0,0)[tl]{$X_5$}}
\put(126.3,37.5){\makebox(0,0)[tl]{$X_3$}}
\put(31.0,131.5){\makebox(0,0)[tl]{$X_1$}}
\put(51.5,46.5){\makebox(0,0)[tl]{$e_6$}}
\put(105.0,117.0){\makebox(0,0)[tl]{$e_5$}}
\put(86.3,60){\makebox(0,0)[tl]{$e_4$}}
\put(107,49.3){\makebox(0,0)[tl]{$e_3$}}
\put(68.8,88.8){\makebox(0,0)[tl]{$e_2$}}
\put(57.0,114.8){\makebox(0,0)[tl]{$e_1$}}
\put(106.3,65.0){\makebox(0,0)[tl]{$u_3$}}
\put(106.3,98.3){\makebox(0,0)[tl]{$u_2$}}
\put(54.0,98.3){\makebox(0,0)[tl]{$u_1$}}
\put(103.8,85.0){\makebox(0,0)[tl]{$f$}}
\put(83.5,106.0){\makebox(0,0)[tl]{$e$}}
\qbezier(144.8,133.5)(154.9,139.5)(165.0,145.5)
\qbezier(136.5,141.8)(140.6,152.6)(144.8,163.5)
\qbezier(33.8,140.3)(27.8,151.1)(21.8,162.0)
\qbezier(24.8,130.5)(13.1,135.8)(1.5,141.0)
\qbezier(144.8,30.8)(155.6,25.5)(166.5,20.3)
\qbezier(137.3,21.8)(141.8,10.9)(146.3,0.0)
\qbezier(29.3,21.0)(23.3,10.5)(17.3,0.0)
\qbezier(21.0,29.3)(10.5,22.9)(0.0,16.5)
\qbezier(102.8,99.0)(113.3,109.5)(123.8,120.0)
\qbezier(66.8,99.0)(56.6,109.1)(46.5,119.3)
\qbezier(102.8,62.3)(112.9,52.1)(123.0,42.0)
\qbezier(57.0,54.8)(49.9,47.6)(42.8,40.5)
\put(132.8,129.8){\circle{25.6}}
\put(37.5,128.3){\circle{25.5}}
\put(33.0,33.0){\circle{25.5}}
\put(132.8,33.8){\circle{25.0}}
\qbezier(79.5,62.3)(91.1,62.3)(102.8,62.3)
\qbezier(66.8,99.0)(66.8,87.8)(66.8,76.5)
\put(66.8,63.8){\circle{25.7}}
{\linethickness{0.8pt}
\qbezier(102.8,99.0)(102.8,80.6)(102.8,62.3)
\qbezier(66.8,99.0)(84.8,99.0)(102.8,99.0)}
\end{picture}\caption{\label{fig3}An illustration for  Case 2.}
\end{figure}

Let $G_1:=G/X_1\ominus e_1/(G[Y_1\cup \{u_2,u_3\}],Y_2,X_3)$, with the order of operation from left to right.
Assign weight 0 to the three vertices resulting from contracting $G[Y_1\cup \{u_2,u_3\}]$, $Y_2$, and $X_3$
and let all other vertices of $G_1$ inherit their weights from $G$.
Then $w(G_1)=x_5+z$. By Claim 3, $G/X_1\ominus e_1$ is 3-connected. Since
$|\partial_{G/X_1\ominus e_1}(G[Y_1\cup
\{u_2,u_3\}])|=|\partial_{G/X_1\ominus
  e_1}(Y_2)|=|\partial_{G/X_1\ominus e_1}(X_3)|=3$, $G_1$ is 3-connected. Hence, Theorem~\ref{thm-main}
holds for $G_1$; so  $G_1$ has a cycle $D_1$ such that $e_3,e_6\in E(D_1)$ and $w(D_1)\geq w(G_1)^r=(x_5+z)^r$.
Note that $D_1$ goes through the vertices representing the
contractions of $X_3,Y_1,Y_2$. So by Lemma~\ref{lem-merge},
there is a cycle $C_1$ in $G$
such that $e,f\in E(C_1)$ and
\begin{equation}\label{eq-aim2-x}
 w(C_1)\geq y_1^r + y_2^r + x_3^r + (x_5+z)^r.
\end{equation}

Now let $G_2:=G/X_5\ominus e_5/(G[Y_1\cup \{u_1,u_3\}],Y_2,X_3)$. 
Assign weight 0 to the three vertices resulting from contracting $G[Y_1\cup \{u_1,u_3\}]$, $Y_2$, and $X_3$,
and let all other vertices of $G_2$ inherit their weights from $G$. By
Claim 3 and the maximality of $X_5$, $G/X_5\ominus e_5$ is
3-connected; so $G_2$ is 3-connected. Hence,
Theorem~\ref{thm-main} holds for $G_2$ and, thus,  $G_2$ has a cycle $D_1$
such that $e_3,e_6\in E(D_1)$ and $w(D_1)\ge w(G_2)^r=(x_1+z)^r$.
Note $D_1$ goes through the vertices representing the contractions of
$Y_1,Y_2,X_3$. Hence, by Lemma~\ref{lem-merge},
there is a cycle $C_2$ in $G$ such that $e,f\in E(C_2)$ and
\begin{equation}\label{eq-aim2-x5}
 w(C_2)\geq y_1^r + y_2^r + x_3^r + (x_1+z)^r.
\end{equation}

Next, let $G_y:=G/(Y_1,Y_2)\ominus e_6/(X_1,G[\{u_1,u_2,u_3\}],X_3)$. 
Assign weight 0 to the three vertices resulting from the contraction of
$X_1$, $G[\{u_1,u_2,u_3\}]$ and $X_3$ and let all other vertices of
$G_y$ inherent their weights from $G$. Then $w(G_y)=x_5+z$.
If $G/(Y_1,Y_2)\ominus e_6$ has an edge-cut $F$ with $|F|\le 2$ then,
since $G/(Y_1,Y_2)$ is 3-connected, $|F|=2$ and $F':=F\cup \{e_6\}$ is
a 3-edge-cut in $G$; but the component of $G-F'$ containing $Y_2$
contradicts the maximality of $Y_2$. Hence,  $G/(Y_1,Y_2)\ominus e_6$ is 3-connected; so
$G_y$ is 3-connected, and Theorem~\ref{thm-main} holds for $G_y$.
Thus $G_y$ has a cycle $D_y$ such that $e_1,e_3\in E(D_y)$ and $w(D_y)\ge w(G_y)^r=(x_5+z)^r$.
Note that $D_1$ going through $e,f$ and the vertices representing the
contractions of $X_1,X_3$. Hence, by Lemma \ref{lem-merge}, there is a cycle $C_y$ in $G$ such that
$e,f\in E(C_y)$ and
\begin{equation}\label{eq-aim2-y}
 w(C_y)\geq x_1^r+x_3^r+(x_5+z)^r.
\end{equation}

Finally, let $G_z:=G/X_5\ominus e_5/(G[Y_1\cup\{u_1,u_3\}],X_1,X_3,Y_2)$.
Assign weight 0 to the four vertices resulting from the contractions of
$G[Y_1\cup\{u_1,u_3\}]$, $X_1$, $X_3$ and $Y_2$, respectively, and let
other vertices of $G_z$ inherit their weights from $G$. Let
$e_1=u_1v_1$, $e_5=u_2v_2$ and $e_3=u_3v_3$, and let
$N_{G_z}(v_i)=\{u_i,z_i,w_i\}$. By the
maximality of $X_1$ and $X_3$, $G_z$ satisfies the conditions of Lemma~\ref{lem-3-edge}.
So by Lemma~\ref{lem-3-edge}, $G_z\ominus uv_2$ contains a cycle
through $e_{51},v_1v_3,e_{31}$, or $G_z\ominus uv_1$ contains a cycle through $e_{51}, v_3v_2, e_{11}$.
Assume the former. Then $G$ contains a cycle $D_z$ through
$e_1,e,f,e_4,e_6, e_{51}, e_{31}$.
By Lemma \ref{lem-merge}, there is a cycle $C_z$ in $G$ such that $e,f\in E(C_z)$ and
\begin{equation}\label{eq-aim2-z}
w(C_z)\geq x_1^r+x_3^r+x_5^r+y_1^r+y_2^r.
\end{equation}
Similarly, such $C_z$ exists if $G\ominus uv_1$ has a cycle through
$e_{51}, v_3v_2, e_{11}$ (with $e_{51}$ playing the role of $e$).

We now  show that $\max\{w(C_1),w(C_2),w(C_y),w(C_z)\}\geq w(G)^r$ from
which the assertion of the lemma holds.
We may assume $x_1>0$ as otherwise $C_1$ is the desired cycle; hence
$x_3>0$ as $x_3\geq x_1$. Similarly, $x_5>0$, $y_1+y_2>0$ and $z>0$; as
otherwise $C_2$, $C_y$ or $C_z$ is the desired cycle.

It is easy to verify that  $w(C_1)-w(G)^r$, $w(C_2)-w(G)^r$,
$w(C_y)-w(G)^r$ and $w(C_z)-w(G)^r$ are increasing functions with respect to $x_3$.
Therefore, since we assume $x_1\leq x_3$, we may further assume that $x_1=x_3$ and denote it by $x$.
Furthermore, by noting that $y_1^r+y_2^r\geq (y_1+y_2)^r$ and writing $y:=y_1+y_2$, it suffices to show
$\max\{f_1,f_2,f_3,f_4\}\geq 0$ for all $x,x_5,y,z\in \mathbb{Z}^+$, where
\begin{align*}
f_1(x,x_5,y,z) &= x^r+y^r+(x_5+z)^r-(2x+x_5+y+z)^r,\\
f_2(x,x_5,y,z) &= x^r+y^r+(x+z)^r-(2x+x_5+y+z)^r, \\
f_3(x,x_5,y,z) &= 2x^r+(x_5+z)^r-(2x+x_5+y+z)^r, \mbox{ and }\\
f_4(x,x_5,y,z) &= 2x^r+x_5^r+y^r-(2x+x_5+y+z)^r.
\end{align*}

Suppose, for a contradiction, that there exist $x,x_5,y,z\in \mathbb{Z}^+$ such that
$f_i(x,x_5,y,z)<0$ for $i=1,2,3,4$.
Let $t=\min\{x,y,x_5\}$. Then $t>0$. Moreover, $z>t$; otherwise, since $f_4$ is
increasing with respect to $x,y,x_5$ and decreasing with respect to $z$,
$f_4(x,x_5,y,z)\geq f_4(t,t,t,t)=4t^r-(5t)^r\approx0.376t^r\geq0$, a contradiction.

If $t=x$ then, since $f_1$ is increasing with respect to $x_5,y,z$,
$f_1(x,x_5,y,z)\geq f_1(t,t,t,t)=(2+2^r-5^r)t^r\approx0.117t^r>0$,
a contradiction. Similarly, if $t=x_5$ then $f_2(x,x_5,y,z)\ge
f_2(t,t,t,t)\approx 0.117t^r> 0$, and
if $t=y$ then $f_3(x,x_5,y,z)\ge f_3(t,t,t,t)\approx 0.117t^r>0$.
\end{proof}

\section{Nonadjacent edges}

In this section we prove Theorem~\ref{thm-main}(b) for graphs of order $n$ under the assumption that Theorem~\ref{thm-main}
holds for graphs of order less than $n$.

\begin{lem}\label{lem-main-b}
Let $n\ge 4$ be an integer and assume that Theorem~\ref{thm-main} holds for graphs of order less than $n$.
Let $G$ be a 2-connected cubic graph of order $n$, let
$w: V(G)\rightarrow \mathbb{Z}^+$, and let $e,f\in E(G)$ such that
$w(V(\{e,f\}))=0$, $e$ and $f$ are nonadjacent, and
every 2-edge cut in $G$ separates $e$ from $f$.
 Then there is a cycle $C$ in $G$ such that $e,f\in E(C)$ and $w(C)\geq cw(G)^r$, where
$r=0.8$ and $c=1/(8^r-6^r)$.
\end{lem}

\begin{proof}

First, we may assume
\medskip
{\it Claim 1}.  $G$ is 3-connected.

For, suppose that there exists a 2-edge cut, say  $F$, in $G$. Then by assumption, $F$ separates $e$ from $f$ in $G$; so
let $A,B$ denote the components of $G-F$ containing $e,f$, respectively. Let $G_1$ be the graph obtained from $A$ by adding an
edge (call it $f'$) between the vertices in $V(F)\cap V(A)$, and let $G_2$ denote the graph obtained from $B$
by adding an edge (call it $e'$) between the vertices in $V(F)\cap V(B)$.

It is easy to see that $G_1$ and $G_2$ are 2-connected, and any 2-edge cut
in $G_1$ (respectively, $G_2$) must separate $e$ (respectively,
$e'$) from $f'$ (respectively, $f$).
Hence Theorem~\ref{thm-main} holds for $G_1$ and $G_2$. So $G_1$ has a cycle $C_1$ such that $e,f'\in E(C_1)$ and
$w(C_1)\ge cw(G_1)^r=cw(A)^r$, and $G_2$ has a cycle $C_2$ such that $e',f\in E(C_2)$ and $w(C_2)\ge cw(G_2)^r=cw(B)^r$.
Now $C:=G[E(C_1-f')\cup E(C_2-e')\cup F]$ is a cycle in $G$ such that $e,f\in E(C)$ and $w(C)\ge cw(A)^r+cw(B)^r\geq cw(G)^r$,
completing the proof of Claim 1.

\medskip

The same argument for Claim 1 in the proof of Lemma~\ref{lem-main-a} can be used here to give

\medskip
{\it Claim 2}. No nontrivial 3-edge cut in $G$  contains $e$ or $f$.

\medskip

We now prove that we may assume

{\it Claim 3}. No 3-edge cut in $G$ separates $e$ from  $f$.

For, suppose $S=\{f_1,f_2,f_3\}$ is a 3-edge cut in $G$ and let $A,B$ denote the components of $G-S$ such that $e\in E(A)$ and
$f\in E(B)$. Clearly $G/A$ and $G/B$ are 3-connected. So Theorem~\ref{thm-main} holds for $G/A$ and $G/B$. Let the new
vertices in $G/A$ and $G/B$ resulting from contractions have weight 0, and let all other vertices of $G/A$ and $G/B$ inherit their weights
from $G$. So for each $1\le i\le 3$, $G/B$ has a cycle $C_i$ such that
$e,f_i\in E(C_i)$ and $w(C_i)\ge cw(G/B)^r=cw(A)^r$, and
$G/A$ has a cycle $C_i'$ such that $f_i,f\in E(C_i')$ and $w(C_i')\ge cw(G/A)^r=cw(B)^r$.

The sets $E(C_i)\cap S$, $1\le i\le 3$,  include at least two of $\{f_1,f_2\}$, $\{f_1,f_3\}$, $\{f_2,f_3\}$, as do the sets $E(C_j')\cap S$, $1\le j\le 3$. Therefore,
there are $i$ and $j$ such that $E(C_i)\cap S=E(C_j')\cap S$. Then $C:=G[E(C_i)\cup E(C_j')]$ is a cycle in $G$
such that $e,f\in E(C)$ and $w(C)\geq cw(A)^r + cw(B)^r \geq cw(G)^r$, completing the proof of Claim 3.

\medskip

Next we show that $e$ and $f$ cannot be very close to each other.  Specifically, we may assume

{\it Claim 4}. There is no edge adjacent to both $e$ and $f$.

For, suppose there is an edge $g=u_1u_2$ such that $u_1\in V(e)$ and $u_2\in V(f)$. Let
$e',f'$ be the edges in $G-\{e,f,g\}$ incident with $u_1,u_2$,
respectively. Then  any 2-edge cut in $G\ominus g$ separates $e=e'$ from $f=f'$,
since $G$ is 3-connected (by Claim 1). So by the assumption of this lemma, Theorem~\ref{thm-main} holds for  $G\ominus g$
and, hence, there is a cycle $C_1$ in  $G\ominus g$ such that $e=e',f=f'\in E(C_1)$ and
$w(C_1)\geq cw(G\ominus g)^r=cw(G)^r$. Now $C:=G[E(C_1)\cup \{e,e',f,f'\}]$
is a cycle in $G$ such that $e,f\in E(C)$ and $w(C)=w(C_1)\geq cw(G)^r$,
completing the proof of Claim 4.

\medskip

We now fix some notation.
Let $e=u_1u_2$ and $f=u_3u_4$. For each $1\le i\le 4$, let $e_{i1},e_{i2}$
denote the edges of $G-\{e,f\}$ incident with $u_i$.
By Claim 4, $e_{ij}\ne e_{kl}$ if $(i,j)\ne (k,l)$.
For each $1\le i\le 4$ and $j=1,2$, let $X_{ij}\subseteq G-\{u_1,u_2,u_3,u_4\}$ be maximal such that $|\partial_G(X_{ij})|=3$
 and $e_{ij}\in \partial_G(X_{ij})$. Note  $X_{ij}$ exists,
as $G$ is cubic and $V(e_{ij})-\{u_i\}$ satisfies these conditions (except
for the maximality). Since $G$ is 3-connected and cubic, each $X_{ij}$ is connected.
Moreover, by the maximality of $X_{ij}$, $X_{ij}$ is uniquely defined. Moreover, for any $i,j$,
\begin{equation}\label{eq-3connected}
 G/X_{ij}\ominus e_{ij} \mbox{ is a 3-connected cubic graph.}
\end{equation}
For, if $F$ is a 2-edge-cut in  $G/X_{ij}\ominus e_{ij}$ then $F':=F\cup \{e_{ij}\}$ is a 3-edge-cut in $G$.
However, $F'$ separates  $e$ from $f$, contradicting Claim 3.

\medskip

{\it Claim 5}.  $X_{11},X_{12},X_{21},X_{22}$ (respectively, $X_{31},X_{32},X_{41},X_{42}$) are pairwise disjoint.

By symmetry, it suffices to show that $X_{11}\cap X_{12}=\emptyset$ and $X_{11}\cap X_{21}=\emptyset$.
If $X_{11}\cap X_{12}\neq\emptyset$ then by Lemma \ref{lem-submod} and by the maximality of $X_{11}$ and $X_{12}$,
we have $X_{11}=X_{12}$; so $\partial_G(X_{11}\cup\{u_1\})$ is a 2-edge cut
in $G$, contradicting Claim 1. If $X_{11}\cap X_{21}\neq\emptyset$ then
by Lemma \ref{lem-submod} and the maximality of $X_{11}$ and $X_{21}$ we have $X_{11}=X_{21}$; so
$\partial_G(X_{11}\cup\{u_1,u_2\})$ is a 3-edge cut in $G$ separating $e$ from $f$, contradicting Claim 3.

\medskip
{\it Claim 6}. If $X_{ij}\cap X_{kl}\neq\emptyset$ for some $i\in\{1,2\}$,
$k\in\{3,4\}$, and $j,l\in \{1,2\}$
then $X_{ij}=X_{kl}$, $G/(X_{i(3-j)},X_{kl})\ominus e_{i(3-j)}\ominus e_{kl}$ is 2-connected
and every 2-edge cut in it separates $e$ from $f$, and  $G/(X_{k(3-l)},X_{ij})\ominus e_{k(3-l)}\ominus e_{ij}$ is 2-connected
and every 2-edge cut in it separates $e$ from $f$.

Note the symmetry in the statement of Claim 6. Suppose Claim 6 fails and,
without loss of generality, let $X_{11}\cap X_{31}\neq\emptyset$. By Lemma \ref{lem-submod} and the
maximality of $X_{11}$ and $X_{31}$, we have $X_{11}=X_{31}$. Let $G':=G/(X_{12},X_{31})\ominus e_{12}\ominus e_{31}$
and $G''=G/(X_{12},X_{31})\ominus e_{12}$. Then $G'=G''\ominus
e_{31}$. Since $G''=G/X_{12}\ominus e_{12}/X_{31}$ and because of
\eqref{eq-3connected},
$G''$ is 3-connected. Hence $G'$ is 2-connected. Since $G''=(G'\oplus e_{31},e,f)$ and $G''$ is 3-connected,
every 2-edge cut in $G'$ must separate $e$ from $f$, completing the proof of Claim 6.

\medskip

 We may assume

{\it Claim 7.}  $\{X_{i1},X_{i2}\}\neq\{X_{j1},X_{j2}\}$ for $1\leq i\neq j\leq 4$.

By Claim 5, we only need to prove Claim 7 for $i\in \{1,2\}$ and $j\in\{3,4\}$. By symmetry,
it suffices to show that $\{X_{11},X_{12}\}\neq\{X_{31},X_{32}\}$.  Suppose $X_{11}=X_{31}$ and $X_{12}=X_{32}$.
Let $\partial_G(X_{11})=\{e_{11},e_{31},f_1\}$ and $\partial_G(X_{12})=\{e_{12},e_{32},f_2\}$.
Since $X_{11}=X_{31}$, it follows from Claim 6 that $G':=G/(X_{12},X_{31})\ominus e_{12}\ominus e_{31}$
is 2-connected and any of its 2-edge cuts  separates $e$ from $f$. So by the assumption of this lemma,
Theorem \ref{thm-main} holds for $G'$. Assign weight 0 to the new vertices resulting from contractions, and
let the other vertices of $G'$ inherit their weights from $G$. Then there is a cycle $C_1$ in $G'$ such that
$e=e_{11}=f_1,f=e_{32}=f_2\in E(C_1)$ and $w(C_1)\geq cw(G')^r$.
Note that $C_1$ goes through the vertices representing the contractions of $X_{11}$ and $X_{12}$.  By Lemma \ref{lem-merge},
there is a cycle $C$ in $G$ such that $e,f\in E(C)$ and
\begin{eqnarray*}
w(C) &\geq & cw(G')^r+w(X_{11})^r+w(X_{12})^r\\
& \geq & c(w(G')+w(X_{11})+w(X_{12}))^r\\
& = & cw(G)^r.
\end{eqnarray*}
This completes the proof of Claim 7.

\medskip

 We may further assume

{\it Claim 8.} $\{X_{i1},X_{i2}\}\neq\{X_{jj'},X_{kk'}\}$ for $i,j',k'\in \{1,2\},j,k\in \{3,4\}$
or $j,j',k,k'\in\{1,2\},i\in \{3,4\}$.

By the assumption of this claim, $i\notin \{j,k\}$.
By Claim 7, we may assume $j\neq k$. So by symmetry, it suffices to show $\{X_{21},X_{22}\}\neq\{X_{31},X_{41}\}$. Suppose the contrary, we may assume by symmetry that  $X_{21}=X_{31}$
and $X_{22}=X_{41}$. Let $X_1:=X_{32},X_2:=X_{42}, Y_1:=X_{21}$ and
$Y_2:=X_{22}$. Further, let $\partial_G(X_1)=\{e_{32},f_{32},g_{32}\}$, $\partial_G(X_2)=\{e_{42},f_{42},g_{42}\}$,
 $\partial_G(Y_1)=\{e_{21},e_{31},h_1\}$, and
 $\partial_G(Y_2)=\{e_{22},e_{41},h_2\}$. Let $x_i:=w(X_i)$ and
 $y_i:=w(Y_i)$ for $i=1,2$, and let $z:=w(G)-x_1-x_2-y_1-y_2$.
Then $w(G)=x_1+x_2+y_1+y_2+z$. By symmetry, we may assume $$y_1\leq y_2.$$

Let $G_y:=G/(Y_1,Y_2)\ominus e_{21}\ominus e_{41}/X_2$. Assign weight 0 to the vertex resulting from the contraction of $X_2$,
and let other vertices of $G_y$ inherit their weights from $G$. Then $w(G_y)=x_1+z$.
By Claim 6 (using $X_{22}=X_{41}$), $G/(Y_1,Y_2)\ominus e_{21}\ominus e_{41}$ (and hence $G_y$) is 2-connected and every 2-edge cut separates
$e$ from $f$; so Theorem \ref{thm-main} holds for $G_y$. Thus, there is a cycle $C_1$ in $G_1$
such that $e=e_{22}=h_2,f=e_{42}\in E(C_1)$ and $w(C_1)\geq c(x_1+z)^r$.
Note that $C_1$ goes through the vertices representing the contractions of $X_2$ and $X_3$. Hence by Lemma~\ref{lem-merge}, there is a
cycle $C_y$ in $G$ such that $e,f\in E(C_y)$ and
$$w(C_y)\geq c(x_1+z)^r + x_2^r + y_2^r.$$
Therefore, we may assume $y_1>0$; as otherwise $C_y$ gives the desired cycle.

Let $G_x:=G/(Y_1,Y_2)\ominus e_{21}\ominus e_{41} /X_2/X_1\ominus e_{32}$.
Assign weight 0 to the vertex resulting from
the contradiction of $X_2$, and let all other vertices inherent their weights from $G$. Then $w(G_x)=z$. By Claim 6 (using $X_{22}=X_{41}$),
$G/(Y_1,Y_2)\ominus e_{21}\ominus e_{41} /X_2$  is 2-connected and each of its 2-edge cuts
separates $e$ from $f$. Thus, $G_x$ is 2-connected. If there is a 2-edge cut $F$ in $G_x$ not separating $e$ from $f$ then
$e,f$ lie in a common component $A$ of $G':=G/(Y_1,Y_2)\ominus e_{21}\ominus e_{41} - (F\cup \{e_{32}\})$. By Claim 6, $F\cup \{e_{32}\}$
is a minimum edge-cut in $G'$. Let $B$ be the component in $G'$ different from  $A$. Then $e_{32}\in \partial_{G}(B)$ and $|\partial_G(B\cup X_{1})|=3$,
contradicting the maximality of $X_1$. So $G_x$ is a 2-connected graph such that each of its 2-edge cuts
must separate $e$ from $f$.
Hence by assumption, Theorem \ref{thm-main} holds
for $G_x$. So there is a cycle $C_1$ in $G_x$
such that $e=e_{22}=h_2,f=e_{42}=e_{31}=h_1\in E(C_1)$ and $w(C_1)\geq cz^r$.
Note that $C_1$ goes through the vertices representing the contractions of $X_2,Y_1,Y_2$. Hence by Lemma \ref{lem-merge}, there is
a cycle $C_x$ in
$G$ such that $e,f\in E(C_x)$ and
$$w(C_x)\geq cz^r  + x_2^r + y_1^r + y_2^r.$$
Therefore, we may assume $x_1>0$; as otherwise $C_x$ gives the desired cycle.

Similarly, by using $X_{21}=X_{31}$ and considering $G_x':=G/(Y_2,Y_1)\ominus e_{22}\ominus
e_{31}/X_1/X_2 \ominus e_{42}$, we obtain a
cycle $C_x'$ in $G$ such that $e,f\in E(C_x')$ and
$$w(C_x')\geq cz^r + x_1 ^r + y_1^r + y_2^r.$$
Thus, we may assume $x_2>0$; as otherwise $C_x'$ gives the desired cycle.

Hence, it suffices to show that $\max\{w(C_x),w(C_x'),w(C_y)\}\geq
cw(G)^r=c(x_1+x_2+y_1+y_2+z)^r$. Since $x_1,x_2,y_1$ are all positive, we see that
$w(C_x)-cw(G)^r$ is increasing with respect to $x_2,y_1,y_2,z$; $w(C_y)-cw(G)^r$ is increasing with respect
to $x_1,x_2,y_2,z$; and $w(C_x')-cw(G)^r$ is increasing with respect to $x_1,y_1,y_2,z$.

Therefore, if $x_1\le x_2$ then $w(C_y)-cw(G)^r\ge cx_1^r+x_1^r+y_1^r-c(2x_1+2y_1)^r$
and $w(C_x)-cw(G)^r\ge x_1^r+2y_1^r-c(2x_1+2y_1)^r$; and if $x_2\le x_1$
then $w(C_y)-cw(G)^r\ge cx_2^r+x_2^r+y_1^r-c(2x_2+2y_1)^r$
and $w(C_x')-cw(G)^r\ge x_2^r+2y_1^r-c(2x_2+2y_1)^r$.
Let
\begin{align*}
  f_1(x,y) &:=x^r +2y^r -c(2x+2y)^r,\\
  f_2(x,y) &:=cx^r +x^r+y^r-c(2x+2y)^r.
\end{align*}
Then it suffices to show $\max\{f_1,f_2\}\geq 0$ for all $x,y\in \mathbb{Z}^+$.
For any $x,y\in \mathbb Z^+$, if $x\leq y$ then by
Lemma \ref{lem-inequality}(v) we have $f_1(x,y)\geq 0$. So we may assume
$x>y$. Then also by Lemma \ref{lem-inequality}(v),
$f_2(x,y)\geq 0$. This completes the proof of Claim 8.

\medskip

Let $q:=|\{(X_{ij},X_{kl}): X_{ij}\cap X_{kl}\ne \emptyset \mbox{ and }
(i,j)\ne (k,l)\}|$. By Claims 5 and 8,
we have $0\leq q\leq 2$. So we have three cases to consider.  Let
$x_{ij}:=w(X_{ij})$ for $1\le i\le 4$ and $j=1,2$.

\medskip
\textit{Case 1.} $q=0$.

 By symmetry, we may assume that
$x_{11}\leq x_{ij}$ for $1\le i\le 4$ and $j=1,2$. Let $\partial_G(X_{12})=\{e_{12},f_{1},f_{2}\}$, and let
$z:=w(G-\bigcup X_{ij})$. Then $w(G)=\sum x_{ij}+z$. (The union and summation are taken over all  $1\le i\le 4$ and $j=1,2$.)

Consider $G_x:=(G/X_{11}\ominus e_{11})/X_{12}$. Assign weight 0 to the vertex resulting from the contraction of $X_{12}$,
 and let other vertices of $G_x$ inherit their  weights from $G$.
Then by \eqref{eq-3connected}, $G_x$ is 3-connected and cubic, and $w(G_x)=w(G)-x_{11}-x_{12}$. So by assumption,
Theorem \ref{thm-main} holds for $G_x$. Thus there is a cycle $C_1$ in $G_x$ such that $e=e_{12},f\in E(C_1)$ and
$w(C_1)\geq cw(G_x)^r$.
Note that $C_1$ goes through the vertex representing the contraction of $X_{12}$. By Lemma \ref{lem-merge}, there is a
 cycle $C$ in $G$ such that $e,f\in E(C)$ and
\[
 w(C)\geq cw(G_x)^r+x_{12}^r=c(x_{21}+x_{22}+x_{31}+x_{32}+x_{41}+x_{42}+z)^r+x_{12}^r.
\]
Since $x_{12}\ge x_{11}$ and $x_{21}+x_{22}+x_{31}+x_{32}+x_{41}+x_{42}+z\geq 6x_{11}$, it follows from
Lemma \ref{lem-inequality}(vi) that $w(C)\geq cw(G)^r$.

\medskip

\textit{Case 2.} $q=1$.

By symmetry, we may assume that $X_{22}\cap X_{32}\neq\emptyset$. Then $X_{22}=X_{32}$ by Lemma \ref{lem-submod} and
the maximality of $X_{22}$ and $X_{32}$. Let $\partial_G(X_{ij})=\{e_{ij},f_{ij},g_{ij}\}$ for
$X_{ij}\notin \{X_{22},X_{32}\}$, let $\partial_G(X_{22})=\{e_{22},e_{32},h_1\}$,
and let $Y:=X_{22}$.
Let  $y:=w(Y_1)=x_{22}=x_{32}$
and $z:=w(G-\bigcup X_{ij})$, where the union is taken over all $1\le i\le 4$ and $j=1,2$. Then
$$w(G)=x_{11}+x_{12}+x_{21}+x_{31}+x_{41}+x_{42} + y +z.$$
By symmetry, we may assume $x_{11}=\min\{x_{11},x_{12},x_{41},x_{42}\}$.
We now find cycles in $G$ and show that one of these is the desired cycle.

Let $G_1:=G/X_{11}\ominus e_{11}/X_{12}$. Assign weight 0 to the vertex resulting
from the contraction of $X_{12}$ and let the other vertices of $G_1$ inherit their weights from $G$. Then
$$w(G_1)=w(G)-x_{11}-x_{12}=x_{21}+x_{31}+x_{41}+x_{42}+y+z.$$
By \eqref{eq-3connected}, $G_1$ is 3-connected. Hence by assumption, Theorem \ref{thm-main} holds for $G_1$.
So there is a cycle $C_1$ in $G_1$ such that $e=e_{12},f\in E(G_1)$
and  $w(C_1)\geq cw(G_1)^r$.
Note that $C_1$ goes through the vertex representing the contraction of $X_{12}$. Hence by Lemma~\ref{lem-merge}, there is
a cycle $C_x$ in $G$ such that $e,f\in E(C_x)$  and
\[
 w(C_x)\geq cw(G_1)^r+x_{12}^r=c(x_{21}+x_{31}+x_{41}+x_{42}+y+z)^r+x_{12}^r.
\]

Let $G_2:=G/(X_{21},X_{32})\ominus
e_{21}\ominus e_{32}/X_{31}$. Assign weight 0
to the vertex resulting from the contraction of $X_{31}$, and let the other vertices of $G_2$ inherit their weights from $G$.
Then $$w(G_2)=x_{11}+x_{12}+x_{41}+x_{42}+z.$$
By Claim 6, $G/(X_{21},X_{32})\ominus e_{21}\ominus e_{32}$ (and hence
$G_2$) is 2-connected and any of its 2-edge cuts  separates $e$ from $f$.
By the assumption of this lemma, Theorem \ref{thm-main} holds for $G_2$. So
there is a cycle $C_1$ in $G_2$ such that $e=e_{22}=h_1,f=e_{31}\in E(C_1)$
and $w(C_1)\geq cw(G_2)^r$.
Note that $C_1$ goes through the vertices representing the contractions of $X_{31},Y$. Hence, by Lemma~\ref{lem-merge}, there is
a cycle $C_x'$ in $G$ such that $e,f\in E(C)$ and
\[
 w(C_x')\geq c(x_{11}+x_{12}+x_{41}+x_{42}+z)^r + x_{31}^r +y^r.
\]

Similarly, by considering $G_2':=G/(X_{31},X_{22})\ominus e_{31}\ominus e_{22}/X_{21}$  we obtain a
cycle $C_x''$ in $G$ such that $e,f\in E(C_x'')$ and
\[
 w(C_x'')\geq c(x_{11}+x_{12}+x_{41}+x_{42}+z)^r + x_{21}^r +y^r.
\]

Finally, let $G_3:=G/Y\ominus e_{22}/X_{21}$. Assign weight 0 to the vertex resulting
from the contraction of $X_{21}$ and let all other vertices of $G_3$ inherit their weights from $G$.
Then $$w(G_3)=x_{11}+x_{12}+x_{31}+x_{41}+x_{42}+z.$$
By \eqref{eq-3connected}, $G_3$ is 3-connected.  By assumption, Theorem \ref{thm-main} holds for $G_3$.
So there is a cycle $C_1$ in $G_3$ such that $e=e_{21},f\in E(C_1)$ and $w(C_1)\geq cw(G_3)^r$.
Note that $C_1$ goes through the vertex representing the contraction of $X_{21}$.  By Lemma \ref{lem-merge}, there is
a cycle $C_y$ in $G$ such that $e,f\in E(C_y)$ and
\[
 w(C_y)\geq  c(x_{11}+x_{12}+x_{31}+x_{41}+x_{42}+z)^r + x_{21}^r.
\]

It suffices to  show that if $x_{21}\le x_{31}$ then $\max\{w(C_x),w(C_x'),w(C_y)\}\geq
cw(G)^r$, and if $x_{31}\le x_{21}$ then  $\max\{w(C_x),w(C_x''),w(C_y)\}\geq
cw(G)^r$.

We may assume $x_{11}>0, x_{21}>0, x_{31}>0, y>0$; otherwise $C_x$ or $C_x'$ or $C_x''$ or $C_y$ is the desired  cycle.
Therefore, $x_{12}>0$, $x_{41}>0$ and $x_{42}>0$.
Thus, it is easy to see that
the functions $w(C_x)-cw(G)^r$, $w(C_x')-cw(G)^r$ and $w(C_y)-cw(G)^r$ are
increasing with respect to each of $x_{12},x_{41},x_{42}$, $x_{31}$ and
$z$; and the function $w(C_x'')-cw(G)^r$ is increasing with respect to each of  $x_{12},x_{41},x_{42}$, $x_{21}$ and $z$.

Suppose $x_{31}\ge x_{21}$.  Since $x_{11}=\min \{x_{11}, x_{12}, x_{41},x_{42}\}$ and $z\ge 0$, we have
\begin{align*}
& w(C_x)-cw(G)^r\ge f_1(x_{11},x_{21},y) := c(2x_{11}+2x_{21}+y)^r+x_{11}^r-c(4x_{11}+2x_{21}+y)^r,\\
& w(C_x')-cw(G)^r\ge f_2(x_{11},x_{21},y): = c(4x_{11})^r + x_{21}^r + y^r -c(4x_{11}+2x_{21}+y)^r,\\
& w(C_y)-cw(G)^r\ge f_3(x_{11},x_{21},y) := c(4x_{11}+x_{21})^r + x_{21}^r -c(4x_{11}+2x_{21}+y)^r.
\end{align*}
For the case when $x_{21}\ge x_{31}$, we use, in the above expressions,  $C_x'',x_{31}$ instead of $C_x',x_{21}$, respectively. So we only consider
the case   $x_{31}\ge x_{21}$.

We now prove that for $x_{11},x_{21},y\in \mathbb{Z}^+$,
$\max\{f_1,f_2,f_3\}\geq 0$. Suppose that this is not true. Then there exist $x_{11},x_{21},y\in \mathbb{Z}^+$
such that $f_i(x_{11},x_{21},y)<0$ for $1\le i\le 3$. Now  $f_1<0$ and Lemma \ref{lem-inequality}(vi) imply
$2x_{11}+2x_{21}+y<6x_{11}$. Hence $2x_{21}+y<4x_{11}.$ Also  $f_2<0$ and Lemma \ref{lem-inequality}(v) imply
$y<x_{21}$ as $4x_{11}>2x_{21}+y>x_{21}$. Finally, $f_3<0$ and Lemma \ref{lem-inequality}(vi) imply
$4x_{11}+x_{21}<6y$ as $x_{21}>y$. Therefore, an easy calculation shows that
\begin{equation}\label{eq-q=1}
 4x_{11}+y<6x_{21}, \quad  x_{21}<2y,  \mbox{ and } 4x_{11}>3y.
\end{equation}
It is easy to see that $f_2$ is increasing with respect to $x_{11}$.
By differentiating $f_2$ with respect to $x_{21}$, we have
$$\frac{\partial f_2}{\partial x_{21}}=\frac{r}{x_{21}^{1-r}}-\frac{2cr}{(4x_{11}+2x_{21}+y)^{1-r}}.$$
So $\partial f_2/\partial x_{21}\geq 0$ if and only if $4x_{11}+2x_{21}+y\geq (2c)^{1/(1-r)}x_{21}\approx 21.275x_{21}$.
Thus, from \eqref{eq-q=1} we see that $\partial f_2/\partial x_{21}<0$. So
$f_2$ is decreasing with respect to $x_{21}$.
Hence, by \eqref{eq-q=1}, we have $$f_2\geq f_2(3y/4,2y,y)=c(3y)^r+(2y)^r+y^r-c(8y)^r=(3^rc+2^r+1-8^rc)y^r\approx0.096y^r> 0,$$
a contradiction.

\medskip
\textit{Case 3.} $q=2$.

By Claims 5, 7 and 8 and by symmetry, we may assume that $X_{12}=X_{32}$ and $X_{22}=X_{42}$.
Let $\partial_G(X_{i1})=\{e_{i1},f_{i1},g_{i1}\}$ for $1\le i\le 4$,
$Y_i:=X_{i2}$ for $i=1,2$, $\partial_G(Y_1)=\{e_{12},e_{32},h_1\}$, $\partial_G(Y_2)=\{e_{22},e_{42},h_2\}$.
Let $x_i:=w(X_{i1})$ for $1\le i\le 4$, $y_i:=w(Y_i)$ for $i=1,2$, and $z:=w(G-\bigcup X_{ij})$ (where the union is taken over all $i,j$).
Then $$w(G)=\sum_{i=1}^4 x_i+y_1+y_2+z.$$
By symmetry, we may assume $x_{1}=\min\{x_{1},x_{2},x_{3},x_{4}\}$.

Let $G_1:=G/(X_{11},Y_1)\ominus e_{11}\ominus e_{32}/X_{31}$.
Assign weight 0 to the vertex resulting from the contraction of $X_{31}$ and let the other vertices of $G_1$ inherit their weights from $G$.
Then $$w(G_1)=x_{2}+x_{4}+y_2+z.$$ Recall that $Y_1=X_{32}$; so by Claim 6, $G/(X_{11},Y_1)\ominus e_{11}\ominus e_{32}$
(and hence $G_1$) satisfies the assumption of our lemma. Therefore, Theorem \ref{thm-main} holds for $G_1$.
So there is a cycle $C_1$ in $G_1$ such that $e=e_{12}=h_1,f=e_{31}\in E(C_1)$ and $w(C_1)\geq cw(G_1)^r$.
Note that $C_1$ goes through the vertices representing the contractions of $X_{31},Y_1$. Hence, by Lemma \ref{lem-merge}, there is
a cycle $C_x$ in $G$ such that $e,f\in E(C_x)$ and
\[
 w(C_x)\geq c(x_{2}+x_{4}+y_2+z)^r + x_3^r+y_1^r.
\]

Let $G_2:=G/Y_1 \ominus e_{12}/X_{11}$. Assign weight 0 to the
vertex resulting from the contraction of $X_{11}$, and let the other vertices of $G_2$ inherit their weights from $G$.
Then by \eqref{eq-3connected}, $G_2$ is 3-connected;  and
$$w(G_2)=w(G)-x_1-y_1=x_2+x_3+x_4+y_2+z.$$
By assumption, Theorem \ref{thm-main} holds for $G_2$.
So there is a cycle $C_1$ in $G_2$ such that $e=e_{11},f\in E(C_1)$ and $w(C_1)\geq cw(G_2)^r$.
Note that $C_1$ goes through the vertex representing the contraction of $X_{11}$. So by Lemma~\ref{lem-merge}, there is
a cycle $C_y$ in  $G$ such that $e,f\in E(C_y)$ and
\[
 w(C_y)\geq c(x_{2}+x_{3}+x_{4}+y_2+z)^r +x_{1}^r.
\]

We now show that $\max\{w(C_x),w(C_y)\}\geq cw(G)^r$. We may assume $x_1>0$ and $y_1>0$; otherwise, $C_x$ or $C_y$ is the
desired cycle. Therefore $x_i\ge x_1>0$ for $i=2,3,4$.
It is easy to see that the functions $w(C_x)-cw(G)^r$ and $w(C_y)-cw(G)^r$ are increasing with respect
to each of $x_{2},x_{3},x_{4},y_2$ and $z$. 
Thus,
\begin{align*}
  &w(C_x)-cw(G)^r\ge f_1(x_1,y_1):=c(2x_1)^r+x_1^r+y_1^r-c(4x_1+y_1)^r,\\
  &w(C_y)-cw(G)^r\ge f_2(x_1,y_1):=c(3x_1)^r+x_1^r-c(4x_1+y_1)^r.
\end{align*}
It suffices to show that $\max\{f_1,f_2\}\geq 0$ for all $x_1,y_1\in \mathbb{Z}^+$. Suppose that this is not true. Then $f_1(x_1,y_1)<0$ and $f_2(x_1,y_1)<0$ for some $x_1,y_1\geq0$.
Since $f_1<0$, it follows from Lemma~\ref{lem-inequality}(v) that $y_1<x_1$. Hence, since $f_2<0$, it follows from
Lemma~\ref{lem-inequality}(vi) that $3x_1<6y_1$. Thus,  $y_1>x_1/2$. Since $f_1$ is increasing
with respect to $y_1$, $f_1\geq f_1(x_1,x_1/2)\geq c(2x_1)^r+x_1^r +(x_1/2)^r -c(9x_1/2)^r\approx 0.109x_{1}^r\geq0$,
a contradiction. \end{proof}

\section{Conclusion}

We can now complete the proof of Theorem \ref{thm-main}. We apply induction
on $|G|$, the order of $G$. If $|G|=2$ then $\{e,f\}$ is contained in a Hamilton cycle;
so Theorem~\ref{thm-main} holds.
Thus we may assume that $|G|\ge 4$ and that Theorem~\ref{thm-main} holds for graphs of
order less than $|G|$. By the remark in Section 1, we may further assume that the weight function satisfies
$w(V(\{e,f\}))=0$ for the  edges $e,f$.
Then by Lemma \ref{lem-main-a}, Theorem~\ref{thm-main}(a) holds for $G$;  and by Lemma \ref{lem-main-b},
Theorem~\ref{thm-main}(b) holds for $G$. \qed

\medskip


The exponent $0.8$ in Theorem~\ref{thm-main} can be increased to the root
of the \myequation,  which is larger than $0.800008$. This can be done by modifying  some calculations, because this root
is the largest number satisfying Lemma \ref{lem-inequality} (more precisely, Lemma \ref{lem-inequality}(i)).

In \cite{bilinski}, the problem for bounding the circumference of cubic
graphs is reduced to one for finding a large Eulerian subgraph
in a 3-edge-connected graph. We take an opposite
approach, and  prove the
following result as a consequence of Theorem~\ref{thm-main}. However, the
result we have is about vertex-weighted graphs while the result in
\cite{bilinski} is about edge-weighted graphs.

\begin{coro}\label{eulerian}
Let $r=0.8$ and let $c=1/(8^r-6^r)\approx 0.922.$
Let $G$ be a 3-edge connected graph,  let $w: V(G)\rightarrow \mathbb{Z}^+$, and $e,f\in E(G)$.
Then $G$ has an Eulerian subgraph $H$ such that
\begin{itemize}
\item [$(a)$] $e\in E(H)$ and  $w(H)\geq w(G)^{r}$, or
\item [$(b)$] $e,f\in E(H)$ and $w(H)\geq cw(G)^{r}$.
\end{itemize}
\end{coro}
\begin{proof} First, we describe a process that constructs a 3-connected cubic graph $L$ from $G$ such that any cycle $C$ in $L$
gives rise to an Eulerian subgraph in $G$ with weight $w(C)$.

\begin{itemize}
\item Pick an arbitrary  $u\in V(G)$ with degree at least 4. If no such vertex exists, let $L:=G$.

\item Suppose $G-u$ is connected. Let $w_i$, $1\le i\le k$,  be
the neighbors of $u$. Let $G_1$ be obtained form $G-u$ by
adding a cycle $v_1v_2\ldots v_kv_1$ (disjoint from $G-u$)
and the edges $w_iv_i$, for $i=1, \ldots, k$. (So no $v_i$ is a cut vertex of  $G_1$ and $G_1$ is 3-edge-connected.)
Extend the weight function $w$ by letting $w(v_1)=w(u)$, and $w(v_i)=0$ for $2\le i\le k$.

\item Suppose $G-u$ is not connected. Let $C_1, \ldots, C_k$ denote the components of $G-u$ and, for each $1\le s\le k$, let
$w_s^1, \ldots, w_s^{n_s}$ be the neighbors of $u$ in $C_s$.
Give each $w_i^j$ a new neighbor $v_i^j$, arrange the $v_i^j$'s in the natural order
$v_1^1v_1^2\ldots v_1^{n_1}v_2^1v_2^2\ldots v_2^{n_2}\ldots v_k^1\ldots v_k^{n_k}$, then swap $v_s^{n_s}$ and $v_{s+1}^1$
for each $s=1,2,\ldots, k$, where $k+1$ is taken as 1, and use this order to construct a cycle on the $v_i^j$'s.
Extend the weight function $w$ by letting $w(v_1^1)=w(u)$, and $w(v_i^j)=0$ for $(i,j)\ne (1,1)$,  $1\le i\le k$ and $1\leq j\leq n_i$.

\item Repeat the above steps for $G_1$, and so on, until we arrive at a cubic graph $L$.
\end{itemize}

From the construction of $L$, we see that each $u\in V(G)$ with degree at least 4 in $G$ corresponds to
a cycle $C_u$. For those edges of $L$ which are not in the cycles $C_u$, we view them as edges of $G$ and use the
same notation in both $G$ and $L$.

We now show that $L$ is 3-connected. Clearly, $L$ is 2-connected and, since $L$ is cubic, it suffices to show that $L$ is
3-edge-connected.
Let $F=\{e_1,e_2\}$ be a 2-edge cut of $L$. Then since $G$ is 3-edge-connected, $F\not\subseteq E(G)$.
For $i\in \{1,2\}$, since $e_{3-i}$ is a cut edge of
$L-e_i$,   if $e_i\in E(G)$ then $e_{3-i}$ cannot lie in any cycle $C_u$, which implies $e_{3-i}\in E(G)$, a contradiction. Hence $e_1,e_2\notin
E(G)$. Since $\{e_1,e_2\}$ is a 2-edge cut in $L$,
$e_1,e_2$ cannot lie in two disjoint cycles in $L$. Thus, both $e_1$ and $e_2$ lie in the cycle $C_u$ in $L$ corresponding
to some $u\in V(G)$; so  $u$ is a cut vertex of $G$. Let $P_1,P_2$ be the components of $C_u-\{e_1,e_2\}$
and, for each $i\in \{1,2\}$, let
$L_i$ denote the component of $L-\{e_1,e_2\}$ containing $P_i$. Since $L$ is cubic, $L_i\neq P_i$. On the other hand,
since $u$ is a cut vertex in $G$ and because of the construction of $L$,
$P_1$ contains some neighbor of $L_2$ and $P_2$ contains some
neighbor of $L_1$, showing that $L-\{e_1,e_2\}$ is connected, a contradiction.
 Hence, $L$ is 3-connected.

Since $e\in E(G)$, $e\in E(L)$. By Theorem~\ref{thm-main}(a), there is a cycle $C'$ in $L$ such
that $e\in E(C')$ and $w(C')\ge w(L)^r$. By contracting all cycles $C_u$ of $L$ corresponding to
vertices $u$ of $G$ back to $u$, we obtain from $C'$ an Eulerian subgraph $H$ such that $e\in E(H)$ and $w(H)\ge
w(G)^r$. Likewise, by Theorem~\ref{thm-main}(b) we find an Eulerian subgraph $H$ such that $e,f\in E(H)$ and $w(H)\ge
cw(G)^r$.
\end{proof}

\medskip

Recall the upper bound $\Theta(n^{\log_98})$ on the circumference of
3-connected cubic graphs,  provided by a construction of
Bondy and Simonovits in \cite{bondy}. The lower bound $\Omega(n^{0.8})$ in
Theorem~\ref{thm-main} is significant in the sense that the exponent breaks
above $\log_43\approx 0.7925$. So we feel that one may be able to improve
the exponent further to $\log_54 \approx 0.861$.

We have used Lemma~\ref{lem-3-edge} to find cycles in cubic graphs through given vertices and edges. There are other results
concerning cycles in cubic graphs which might be useful for further
improving the bound given in Theorem~\ref{thm-main}.
We refer the reader to \cite{aldred2,aldred3,bau}.

\bigskip
{\sc Acknowledgment}. We thank the anonymous referee for a through reading of the original manuscript and very detailed
and helpful suggestions.

\end{document}